\documentclass[12pt, a4paper,  reqno, english]{amsart}
\usepackage[all]{xy}
\usepackage{amssymb,amsmath,amsthm}
\numberwithin{equation}{section}
\usepackage[margin=1in]{geometry}

\newcommand{\N}{\mathbb{N}}

\newtheorem{thm}{Theorem} 
\newtheorem{prop}{Proposition}
\newtheorem{lem}{Lemma}
\newtheorem{cor}{Corollary}

\theoremstyle{definition}

\renewcommand{\mod}[1]{\hspace{-2.9mm}\pmod{#1}}

\newcommand{\e}{{\rm e}}

\newcommand{\ben}{\begin{enumerate}}
\newcommand{\een}{\end{enumerate}}
\newcommand{\eit}{\begin{itemize}}
\newcommand{\beq}{\begin{equation}}
\newcommand{\eeq}{\end{equation}}

\newcommand{\cal}{\mathcal}
\newcommand{\lab}{\label}

\newcommand{\colt}[2]{\genfrac{}{}{0pt}{1}{#1}{#2}}
\newcommand{\tolt}[3]{\colt{#1}{\colt{#2}{#3}}}

\usepackage{color}
\definecolor{red}{rgb}{1,0,0}
\definecolor{blue}{rgb}{.2,.6,.75}
\definecolor{green}{rgb}{.4,.7,.4}

\renewcommand{\leq}{\leqslant}
\renewcommand{\geq}{\geqslant}
\renewcommand{\d}{\mathrm{d}}

\renewcommand{\d}{\mathrm{d}}

\begin{document}

\title{Exponential sums with multiplicative coefficients and applications}

\author{R\'egis de la Bret\`eche and Andrew Granville}
\date{}

\address{
Institut de Math\'ematiques de Jussieu-Paris Rive Gauche\\
Universit\'e de Paris, Sorbonne Universit\'e, CNRS UMR 7586\\
Case Postale 7012\\
F-75251 Paris CEDEX 13\\ France} 
\email{regis.delabreteche@imj-prg.fr} 
\address{D{\'e}partment  de Math{\'e}matiques et Statistique,   Universit{\'e} de
Montr{\'e}al, CP 6128 succ Centre-Ville, Montr{\'e}al, QC  H3C 3J7, Canada; and 
Department of Mathematics, University College London, Gower Street, London WC1E 6BT, England.
}
\email{andrew@dms.umontreal.ca} 

\thanks{  
A.G.~est partiellement soutenu par une bour\-se  de la
Conseil de recherches en sciences naturelles et en g\' enie du Canada, and is also partially supported a
European Research Council  grant, agreement n$^{\text{o}}$ 670239. 
Part of this work was undertaken while A.G.~was in Paris, his stay was supported by Universit\'e Paris--Diderot,
Universit\'e Pierre et Marie Curie, Universit\'e d'Orsay  and Fondation
Sciences Math\'ematiques de Paris. The hospitality and financial support of these
institutions is gratefully acknowledged.
We thank J\"org Br\"udern, Oleksiy Klurman,  Lilian  Matthiesen,  Trevor Wooley and particularly K.~Soundararajan for useful discussions and
 correspondence.
 }

\begin{abstract}
We show that if an exponential sum  with multiplicative
coefficients is large then the associated multiplicative function is
``pretentious''. This leads to applications in the circle method, and a natural interpretation of the local-global principle.
\end{abstract}

\maketitle

\newcommand{\cbar}{\overline{\chi}}
\newcommand{\pbar}{\overline{\psi}}
\newcommand{\sumstar}{\sideset\and ^* \to \sum}

\section{Introduction}
Diverse investigations in analytic number theory involve sums like
$$
R_f(\alpha, x):=\sum_{n\leq x} f(n)\e(n\alpha)
$$
where $\e(t)=\e^{2i\pi t}$ for $t\in \mathbb R$  and $f $ is a multiplicative function. For simplicity we will restrict our attention throughout to the class $\mathcal M$ of completely multiplicative functions $f$ for which $|f(n)|\leq 1$ for all $n$; and let 
\[ F(s)=\sum_{n\geq 1} \frac{f(n)}{n^s}.\]
One can approximate any $\alpha\in \mathbb R/\mathbb Z$ by a rational $\tfrac aq$ with $(a,q)=1$ and $q\leq Q$ (with, say, $Q= x/(\log x)^{2+\varepsilon}$), so that 
\[
\left| \alpha - \frac aq\right| < \frac 1{q^2}.
\]
If $q>(\log x)^{2+\varepsilon}$ then $\alpha$ is on a \emph{minor arc} and Montgomery and Vaughan  \cite{MV77}  proved  that 
\begin{equation}\label{eq.Bachmann}
\sum_{n\leq x} f(n)\e(n\alpha) \ll \frac x{\log x}.
\end{equation}
There are many examples of $f$ that attain this bound.\footnote{For example, no matter what $f(p)$ equals on the primes  $p\leq \tfrac x2$ we select an angle $\theta$ and   $f(p)=\e(\theta-p\alpha)$ on the primes $p, \tfrac 12x<p\leq x$, so that $\e(\theta)$ points in the same direction as the sum of $f(n)$ over integers $n\leq x$ free of prime factors $>\tfrac 12x$.}  In general we have the folklore conjecture
\begin{equation}\label{1.2}
\sum_{n\leq x} f(n)\e(n\alpha) \ll \frac{x}{\log x}+  \frac {x}{\sqrt q} .
\end{equation}
The $\tfrac x{\sqrt{q}}$-term cannot, in general, be removed since there are many examples for which $R_f(\alpha, x) \gg \tfrac x{\sqrt{q}}$ when $q$ is small.
Our main goal in this paper is to classify these examples and to determine asymptotic formulae for $R_f(\alpha, x)$ in such cases. We will express $R_f(\alpha, x)$ in terms of other quantities that arise naturally in multiplicative number theory:

\begin{thm}\label{thm0} Let $\varepsilon >0$, $f\in \mathcal M$, $x\geq 3$ and $\alpha=a/q + \beta$ where $(a,q)=1$ with $q\leq (\log x)^{2+\varepsilon}$. There exists a primitive Dirichlet character $\chi \pmod r$ where $r$ divides $q$, and a real number $t$ with $|t|<\log x$ for which  
 \[
\sum_{n\leq x} f(n)\e(n\alpha)  = \frac{\cbar(a)\kappa(q)  g(\chi)}{\phi(q)}         I(x, \beta,t)   \sum_{n\leq x} \frac{f(n) \cbar (n)}{n^{it}}  +O\bigg(   \frac{ (1+|\beta| x) x}{ (\log x)^{ 1 - \frac{1}{\sqrt{2}} +o(1)} } \bigg)  ,
 \]
 where $g(\chi)$ is the Gauss sum, $\kappa$ is  defined by the convolution  $f(n)/n^{it} = (\kappa*\chi)(n)$, and we take
  \begin{equation}\label{defI}
I(x,\beta,t) := \frac 1x \int_{0}^x \e(\beta v) v^{it} \d v  .
\end{equation}
 \end{thm}
 
 The character $\chi$ and the real number $t$ are selected to maximize the sum on the right-hand side. If this sum remains larger than the error term for $x$  in some range (like from $X$ to $X^2$) then there is a unique possibility for  $\chi$ which does not change as $x$ varies, and $t$ varies continuously if at all.

If $r=q$   with $f(n)=\chi(n)n^{it}$ when $(n,q)=1$ then the main term here is 
\[
\frac{\cbar(a)f(q) q^{-it}  \sqrt{q}}{\phi(q)}      I(x, \beta,t)    \sum_{\substack{n\leq x \\(n,q)=1}} 1\ \sim \  \cbar(a)f(q)q^{-it}  \cdot  I(x, \beta,t)   \cdot   \frac x{\sqrt{q}}.
\]
Since (trivially) $|\cbar(a)f(q)q^{-it}|,\ |I(x, \beta,t)|\leq 1$, this supports the folklore conjecture
\eqref{1.2}. The improved bound 
\[
|I(x,\beta,t)| \ll  \frac 1 {\sqrt{1+|\beta| x}}
\]
proved in \eqref{eq: Iintegral},
suggests our refined conjecture,
\begin{equation}\label{1.2b}
\ R_f(\alpha, x)\ll \frac{x}{\log x}+  \frac {x}{ \sqrt{ q(1+|\beta| x)}} .
\end{equation}

In Theorem \ref{thm1} we will prove rather more  than Theorem \ref{thm0}, obtaining an asymptotic series (of similar looking terms) with a better error term.

In the proof of Theorem  \ref{thm0}, we write each $\e(n\alpha)=\e(\tfrac {an}q)\e(n\beta)$ and replace the $\e(\tfrac {an}q)$ by a sum over characters mod $q$; the same idea works for any bounded function of period $q$. Thus, for example, we also prove that if $\xi$ is a character modulo squarefree $q$ and   $(abc,q)=1$   then 
  there exists a constant $c_q$, which depends on $f,\xi,a,b,c$  with $|c_q|\ll \e^{O(\omega(q))} \frac {q^2}{\phi(q)^2}$, such that  
  \begin{equation}\label{1.Gen}
\sum_{\substack{n\leq x \\ (n,q)=1}} f(n) \xi(n+c) \e\left( \frac{an+b\overline{n}}q \right)
 = \frac{c_q}{\sqrt{r}}   \frac{x^{it}}{1+it}  \sum_{n\leq x} \frac{f(n) \cbar (n)}{n^{it}}  +O\bigg(  \frac{x}{(\log x)^{ 1 - \frac{1}{\sqrt{2}} +o(1)}}  \bigg)   
\end{equation}
with $\chi$ and $t$ as in Theorem \ref{thm0}. With additional care one can prove a version of this result in a range like $q\leq x^{1/2}$.
We prove several more general and precise results in    section \ref{sec: twist}.

Obtaining good estimates for $R_f(\alpha, x)$ on the major arcs allows us to obtain asymptotics in various Diophantine problems (weighted by multiplicative functions) using the circle method. The main term in these asymptotics  typically involve an Euler product that can be decomposed into contributions from each prime (a \emph{local-global principle}). 

In the questions  here the roles of small and large prime factors in the asymptiotic formulae are quite different and so we begin by splitting any   $f\in \mathcal M$ into two multiplicative functions
$f=F_sF_\ell$ where $F_s$ involves only the ``small'' prime factors, and $F_\ell$ only the ``large'', and we define
\[
F_s(p)=\begin{cases}  f(p)& \text{ for } p\leq z ,\\
\chi(p)p^{it} & \text{ for }  p>z, \end{cases}
\]
where $\chi$ and $t$ are defined as above, and we will take $z\asymp \log x$.
Throughout we let $\eta=1-\tfrac 2\pi$ and $\tau:=\tfrac{2-\sqrt{2}}3$.
 
 \begin{cor} [Corollary to Theorem \ref{thm2}] Let $\varepsilon >0$. 
Suppose that $A$ and $B$ are sets of positive integers for which $A+B\subset \{ 1,2,\ldots, x\}$ with 
$|A||B|\geq x^2/(\log x)^{\tau-\varepsilon}$. Then 
\[
\frac 1 {|A||B|}  \sum_{\colt{a\in A}{b\in B}} f(a+b)   = 
\frac 1x \sum_{n\leq x} F_\ell (n)  \cdot \frac 1 {|A||B|}  \sum_{\colt{a\in A}{b\in B}} F_s(a+b)
+ o(1) .
\]
\end{cor}

The mean value of multiplicative functions like $F_\ell (n)$ over $n\leq x$ is well-explored in the literature.
The mean value of $F_s(a+b)$ over $a\in A, b\in B$,  can be given explicitly (see Lemma \ref{lemcor1} below)
since it only varies unpredictably on the very small primes.

The circle method has typically been used to count solutions to  equations in enough variables. 
Here we give solutions in two well-known problems in three variables.

\begin{thm}\lab{thm3} Let $f,g,h\in \mathcal M$.  Given positive integers $a,b,c$, for $x\geq 2$, we have
\begin{align*} \frac 1{x^2/2} \sum_{\colt{\ell,m,n\leq x}{a\ell+bm=cn}}  f(\ell)&g(m)h(n)=
\frac 1x \sum_{n\leq x} F_\ell (n) \cdot \frac 1x \sum_{n\leq x} G_\ell(n) \cdot \frac 1x
\sum_{n\leq x} H_\ell(n)
\cr&\times  \frac 1{x^2/2} \sum_{\colt{\ell,m,n\leq
x}{a\ell+bm=cn}} F_s(\ell)G_s(m)H_s(n) 
+ O \left(   \frac {1}{ (\log x)^{\tau /2+o(1)} }   \right)     .
\end{align*} 
The main term is $o(1)$ unless $\chi_f=\chi_g=\chi_h=1$.
\end{thm}

\begin{thm}\label{thm4} Let $f,g,h\in \mathcal M$ and let $x=N\in\N_{\geq 2}$. Then
\begin{align*} \frac 1{N^2/2}  \sum_{\colt{\ell,m,n\geq 1}{\ell+m+n=N}} f(\ell)&g(m)h(n)   =
\frac 1N \sum_{n\leq N} F_\ell (n) \cdot \frac 1N \sum_{n\leq N} G_\ell(n) \cdot \frac 1N
\sum_{n\leq N} H_\ell(n)
\cr&\times  \frac 1{N^2/2} \sum_{\colt{\ell,m,n\geq 1}{\ell+m+n=N}}
F_s(\ell)G_s(m)H_s(n)   + O \left(  \frac {1}{ (\log x)^{\tau /2+o(1)} }   \right) .
\end{align*}
The main term is $o(1)$ unless $\chi_f=\chi_g=\chi_h=1$.
\end{thm}

The mean values of $F_s(\ell)G_s(m)H_s(n)$ over solutions to $a\ell+bm=cn$ or to
$\ell+m+n=N$, can also be estimated by elementary methods as we will see in section   \ref{EvalCsts}.
For example, suppose that  $A, B$ and $C$ are the  positive
integers generated by  given sets of primes with characteristic functions  $f, g$ and $h$, respectively. If
 $\delta_A:=(1/x) \#\{ a\leq x: a\in A\}$, and similarly $ \delta_B$ and $ \delta_C$
then Theorems \ref{thm3} and \ref{thm4} imply that 
\begin{equation} \label{ABC1}\begin{split}
\frac 1 {x^2/2}&  \#\{(\ell,m,n)\in
A \times B \times C\,:\, \ell+m=n\leq x \}\cr& =
\delta_A\delta_B
\delta_C
\prod_{p \not\in A \cup B \cup C} \left( 1-\frac 1{(p-1)^2} \right)  +o(1) ,
\end{split}\end{equation}
and 
\begin{equation}  \label{ABC2}\begin{split}
& \frac 1 {N^2/2}   \#\{(\ell,m,n)\in
A \times B \times C\,:\, \ell+m+n=N \}\cr&\ \  =  \delta_A\delta_B  \delta_C
\prod_{\substack{p \not\in A \cup B \cup C\\ p\nmid N} } \left( 1 + \frac{  1}{(p-1)^3} \right) 
\prod_{\substack{p \not\in A \cup B \cup C\\ p|N} } \left( 1-\frac 1{(p-1)^2} \right)  +o(1) .
\end{split}\end{equation}

It was shown in \cite{GS01}  that if $f$ is a totally multiplicative function
that only takes values 1 and $-1$ then there are at least
$\tfrac 12(1-\delta_0+o(1))x $ solutions to $f(n)=1$ with $n\leq x$
(and so no more than $\tfrac 12(1+\delta_0+o(1))x $ solutions to $f(n)=-1$), and that
this is best possible (by taking $f(p)=1$ for $p\leq x^{1/(1+\sqrt{e})}$ and $f(p)=-1$ otherwise), where
$$
\delta_0 = -1+2\log(1+\sqrt{\e}) - 4\int_1^{\sqrt{\e}} \frac{\log
t}{t+1} \d t =   0.656999 \ldots.
$$
What about $f(a),f(b),f(c)$ for solutions to $a+b=c$?  We apply Theorem \ref{thm3} to prove the following inequalities.

 \begin{cor}\lab{cor2}  If $f,g,h\in \mathcal M$, taking   only the values 1 and $-1$, then when $x$ tends to $\infty$, 
 \[
 \#\{ 1\leq a,b,c\leq x: \ a+b=c \text{ and } f(a)=g(b)=h(c)=-1\}  \leq   \tfrac 12(\kappa +o(1)) x^2 
 \]
where $\kappa =\tfrac 18(1+\delta_0)^3 =.56869\dots$, and
\[
 \#\{ 1\leq a,b,c\leq x: \ a+b=c \text{ and } f(a)=g(b)=h(c)=1\}  \geq\tfrac 12(\kappa' +o(1)) x^2 
 \]
where $\kappa'=\tfrac 18(1-\delta_0)^3 =.005044\dots$.
\end{cor}

We use this result  to bound  the number of
Pythagorean triples mod $p$ up to any given point:\footnote{Many thanks to Ben Green for suggesting this problem.}\  a proportion of at least $\kappa'$ of the triples of  residues $a,b,c \pmod p$ with $1\leq a,b,c\leq x<p$ and $a+b\equiv c\pmod p$, are all quadratic residues mod $p$; moreover this proportion can be attained for some primes $p$, no matter how large $x$ is.

The organization of this paper is a little complicated as there are lots of strands to bring together. In section 2 we will discuss what is known about mean values of multiplicative functions that is relevant to this paper and state   the more general Theorem \ref{thm1} in this context, from which Theorem \ref{thm0} is deduced.

 \subsection{More on $R_f(\alpha, x)$}
 
 The best general bound for $R_f(\alpha, x)$ in the literature was given by 
Bachmann (\cite{B03b},   Theorem \ref{thm4}): If $|\alpha-\tfrac aq|\leq \tfrac 1{q^{2}}$ with $(a,q)=1,$ then
\begin{equation}  \label{1.1}
R_f(\alpha, x)\ll \frac{x}{\log x}+x\left( \frac {\log R \log\log R}{ R} \right)^{1/2}
\end{equation}
where $R=\min\{q, x/q\}$. One can easily deduce 
\eqref{eq.Bachmann} (see the proof of Proposition \ref{NewBd}  in section \ref{MVBound})
except, perhaps, if   $q\leq Q_1$ where
\begin{equation}\label{defQ1} 
  Q_1:= (\log x)^2(\log\log x)^{1+\varepsilon} \text{ and } 
 \bigg|\alpha - \frac aq \bigg| \leq \frac {Q_1}{qx}
\end{equation}

 In fact  La Bret\`eche (\cite{dlB98}, Proposition 1) showed that if
$R_f(\alpha, x)\gg x/\log x$  and
$f$ is not close to any $\psi(n) n^{it}$,  
then the $f(p)$ are
suitably correlated for enough  primes $p\gg x/\log x$    (see section \ref{friBd} for further details).

We establish in section \ref{MVBound} that  if \eqref{1.2} does not hold then  
\begin{equation}\label{qrange}
(\log x)^{ 2 +o(1) }\leq q\leq (\log x)^2(\log\log x)^{1+o(1)},\ \
\bigg|\alpha - \frac aq \bigg| \leq \frac{\log q\log\log q}x . 
\end{equation}

 \section{Known results on multiplicative functions}
 Let $t_f(x,T)$ denote a value of $t$ which yields the 
  maximum of
\begin{equation} \label{Fmax}
\bigg|F\bigg(1+\frac 1{\log x} + it \bigg)\bigg|  
\end{equation}
as $t$ runs through real numbers with $|t|\leq T$. 
Hal\'asz's Theorem (see \cite{H68}, \cite{GHS}, \cite{T15} and \cite{GSBook})  gives upper bounds for $|\sum_{n\leq x} f(n)|$ in terms of the maximum of \eqref{Fmax} where $t=t_f(x,\log x)$.  In Corollary 2.9.1 of \cite{GSBook}, it is observed that if $t=t_f(x,\log x)$ then
\begin{equation} \label{lemsumt} 
\sum_{n\leq x} f(n) = \frac {x^{it}} {1+it} \sum_{n\leq x} \frac{f(n)}{n^{it}}+ O\left(  x \frac{(\log\log x)^2 }{ (\log x)^{\eta}}   \right).
\end{equation}

If $1\leq w\leq (\log x)^{O(1)}$,  Theorem  1.5 of \cite{GHS}  (improving \cite{E89}) gives
\begin{equation} \label{adaptTh4}
  \sum_{n\leq x/w}f(n) = \frac 1{w^{1+it}}  \sum_{n\leq x} f(n)
+O \left( \frac xw   \frac{(\log\log x)^2}{(\log x)^{ \eta} }  \right) .
\end{equation}

For a Dirichlet character $\chi \bmod q$ we define  
\[
S_f(x,\chi):= \sum_{n\leq x} f(n) \cbar (n) ,
\]
and let $S_f(x)=S_f(x,1)=\sum_{n\leq x} f(n)$.
  We deduce the following from  \eqref{adaptTh4}:
 
 \begin{lem}
\label{corComparePC}Let $f \in \mathcal M$.
If  $\psi \pmod r$ induces $\chi \pmod q$ and  $q, \ell \leq Q_1$ then, for $x\geq 3$ and $t=t_{f \pbar}(x,\log x)$,  
\[
S_f(x,\chi)  =   \frac { I(x,0,t)}{\ell^{1+it}}  \prod_{p|q} \left( 1-\frac{ f(p)\pbar(p)}{p^{1+it}}\right)
\sum_{n\leq x} \frac{f(n)\pbar(n)}{n^{it} }
+O \left( \frac{q/r}{\phi(q/r)}   \, \frac{x(\log\log x)^2 }{ \ell (\log x)^{\eta}}  \right).
\]
\end{lem}

\begin{proof} Let $q_r:=\prod_{p|q, p\nmid r} p$.   We have the identity
\begin{align*}
S_f(x/\ell,\chi) &= \sum_{\colt{n\leq x/\ell }{(n,q_r)=1}} (f\pbar)(n) =  \sum_{d|q_r} \mu(d) \sum_{\colt{n\leq x/\ell }{d|n}} (f\pbar)(n)
\cr&=  \sum_{d|q_r} \mu(d) (f\pbar)(d)  S_f(x/d\ell,\psi).
\end{align*}
By   \eqref{adaptTh4} and then \eqref{lemsumt},  we have
$$
S_f(x/d\ell,\psi) =  \frac 1{(d\ell)^{1+it}} \cdot   I(x,0,t) \sum_{n\leq x} \frac{f(n)\pbar(n)}{n^{it}}    +O \left(  \frac x{d\ell} \, \frac{(\log\log x)^2 }{ (\log x)^{\eta}}  \right) .
$$
Substituting this in above yields the claim.
\end{proof}

\subsection{Mean values of multiplicative functions in arithmetic progressions} \label{sec: MVchars}
When $(a,q)=1$, we have the usual decomposition for a sequence in arithmetic progression:
\begin{equation} \label{apdecomp} 
\sum_{\substack {n\leq x \\ n\equiv a \mod q}} f(n)  = \frac 1{\phi(q)} \sum_{\chi \mod q} \chi (a) S_f(x,\chi) .
\end{equation}
To determine the largest summands on the right-hand side of \eqref{apdecomp} for a range of  $x$, define
\[
s_f(X,\chi):= \max_{\sqrt{X} \leq x \leq X^2} \frac{|S_f(x,\chi)|}{x} ,
\]
and then order the characters mod $q$ as $\chi_1, \chi_2,\ldots$ so that 
\[ s_f(X,\chi_1)\geq s_f(X,\chi_2)\geq s_f(X,\chi_3)\geq \ldots \]
The first part of Theorem 1.8, together with Theorem 1.9  from  \cite{GHS} then implies that if $q\leq Q_1^2$  then
for any fixed $J\geq 2$, and all $x$ in the range   $\sqrt{X} \leq x \leq X^2$ we have 
\begin{equation} \label{PLS} 
  \sum_{\substack{\chi \mod q  \\ \chi\not \in \{ \chi_1,\ldots, \chi_{J-1} \} }}	|S_f(x,\chi) |^2
  \ll_{J}   
   \Bigg( \frac{x(\log\log x)^2}{(\log x)^{ 1 - \frac{1}{\sqrt{J}}  }} \Bigg)^2   .
\end{equation}
Theorem 1.8  from  \cite{GHS} also gives that 
 \begin{equation} \label{apestimate} 
\sum_{\substack {n\leq x \\ n\equiv a \mod q}} f(n)  = \frac 1{\phi(q)} \sum_{j=1}^{J-1} \chi_j (a) S_f(x,\chi_j) +
O\left(  \frac x{\phi(q)} \, \frac{(\log\log x)^2}{(\log x)^{ 1 - \frac{1}{\sqrt{J}}  }}\right) .
\end{equation}
Precursors to this result may be found in the work of Elliott \cite{E02}.

 Let $\psi_j \pmod {r_j}$ be the primitive character that induces $\chi_j$ for each $j\geq 1$. Let $t_j=t_{f\pbar_j}(x,\log x)$ and define the multiplicative function
   \[
f_j(n):=f(n)\pbar_j(n)n^{-it_j}\in \mathcal M \text{ for each } j.
\]
(We will sometimes suppress the subscript ``$j$'' and write $f_*$ in place of $f_j$.)
Therefore by Lemma \ref{corComparePC} with $\ell=1$,  and taking $J\geq 3$  in \eqref{apestimate} (as $1-\frac{1}{\sqrt{3}}>\eta$) we obtain
\begin{equation} \label{apsumpsi}
\sum_{\substack {n\leq x \\ n\equiv a \mod q}}\!\!\!\!\!\!\!\! f(n)  = \frac 1{\phi(q)} \sum_{j=1}^{J-1} \psi_j (a)   k_j(q ) \, I(x,t_j)  \,S_{f_j}(x)   
+O \left( \frac{x}{\phi(q)}   \, \frac{(\log\log x)^{2+o(1)} }{ (\log x)^{\eta}}  \right)
\end{equation}
where $k_j$ is the multiplicative function with $k_j(p^a) := 1 - f_j(p) /p$.

\section{A certain twisted integral} \label{sec: twisted}

 We need estimates for $I(x,\beta,t) := \frac 1x \int_{0}^x \e(\beta v) v^{it} \d v $. Evidently
 \[
 I(x,0,t)=\frac{x^{it}}{1+it} \text{ and } I(x,\beta,0) =   \frac{\e(\beta x)-1  } {2i\pi \beta x},
 \]
 and every $|I(x,\beta,t)|\leq 1$ as $| \e(\beta v) v^{it}|=1$. To bound $I(x,\beta,t)$ in general we use the  stationary phase method, 
 writing $I(x,\beta,t)=\frac 1x \int_{0}^x \e^{iF(v)} \d v $ where $F(v)=2\pi \beta v+t\log v$  is a real thrice-differentiable function for all $v>0$.  For any interval $0\leq a<b\leq x$, Lemmas 4.2 and 4.4 of \cite{ECT} imply that
 \[
 \int_{a}^b \e^{iF(v)} \d v \ll\min \Bigg\{ \frac 1{\min_{v\in [a,b]} |F'(v)|},\ \frac 1{\min_{v\in [a,b]} |F''(v)|^{1/2}}\Bigg\}
 \]
and when moreover when   $c:=\tfrac {-t}{2\pi \beta} \in [a,b]$ (here $c$ is selected so that $F'(c)=0$)  and $a\asymp b\asymp c$,
\[
 \int_{a}^b \e^{iF(v)} \d v \ll  
 \frac{1+\sqrt{|t|}}{|\beta|}.
 \]

 The second inequality always gives $I(x,\beta,t)\ll 1/\sqrt{|t|}$.
 If $\beta$ and $t$ have the same sign or $|t|>3\pi |\beta| x$ then the first inequality yields 
 $I(x,\beta,t)\ll 1/|\beta| x$. Otherwise we use the third inequality for the interval $[\tfrac c2,\min\{ x,2c\}]$, and the first inequality for the rest of $[0,x]$. Collecting this together implies that 
\begin{equation} \label{eq: Iintegral}
|I(x,\beta,t)| \ll \min \left\{ 1,\frac 1{\sqrt{|t|}}, \ \frac {1+\sqrt{|t|}}{|\beta| x}  \right\} \ll \frac 1 {\sqrt{1+|\beta| x}} \, .
\end{equation}

Taking  $v=xw$ and $\gamma=x\beta$ in the definition of $ I(x,\beta,t)$, we obtain
\[
I(x,\beta,t)=x^{it} I(1,\gamma,t) \text { where } I(1,\gamma,t)  =\int_{0}^1 \e(\gamma w) w^{it} \d w= \widehat h_t(-\gamma) 
\]
with $h_t(w)=w^{it}$ for $0\leq w\leq 1$, and $h_t(w)=0$ otherwise.
This implies that 
\[
x \int_{
-\Delta/x}^{\Delta/x}  | I(x,\beta,t) |^2  \d \beta  =  \int_{
-\Delta}^{\Delta}   | \widehat h(\gamma)  |^2 \d\gamma 
\]
 By Plancherel's Theorem, we see this is bounded by $1$ since
 \[
      \int_{
-\infty}^{\infty}  |\widehat h(\gamma)|^2 \d \gamma = \int_{
-\infty}^{\infty}  |  h(w)|^2 \d w = 1.
 \]
 By \eqref{eq: Iintegral}, we have $|\widehat h(\gamma)|^2\ll (1+|t|)/\gamma^2$ and so
\begin{equation} \label{eq: Iintegralbd}
x \int_{
-\Delta/x}^{\Delta/x}  | I(x,\beta,t) |^2  \d \beta  =  \int_{
-\Delta}^{\Delta}   | \widehat h(\gamma)  |^2 = 1 +O\bigg( \frac {1+|t|}\Delta \bigg) .
\end{equation}

\section{Exponential sums with multiplicative coefficients} \label{mainthmsec}

If $(b,q)=1$ then
 \begin{equation}\label{ebq} 
 \e(b/q) = \frac 1{\phi(q)} \sum_{\chi\, \mod q} \cbar(b) g(\chi)
 \end{equation}
 where  we define the \emph{Gauss sum} to be
\[
g(\chi)=\sum_{m=1}^{q-1} \chi(m) \e\left(\frac mq\right).
\]
A complication arises when $(b,q)>1$. For example if $q$ is prime then $b\equiv 0 \pmod q$ and  $\e(b/q) =1$ (as $q$ is prime), we use a quite different formula. This explains the main technical difficulty in proving our  main theorem:

\begin{thm}\label{thm1}   Let $f\in \mathcal M$, $x\geq 3$  and $\alpha=a/q + \beta$ where $(a,q)=1$ with $q\leq Q_1$,  defined as in \eqref{defQ1}, with all the assumptions and notation given above. For any integer $J\geq 2$ we have
 \begin{align*} 
\sum_{n\leq x} f(n)\e(n\alpha)  = \frac{1}{\phi(q)}  \sum_{\substack{j=1 }}^{J-1} \pbar_j(a)  g(\psi_j)  & \ \kappa_j(q/r_j)   I(x, \beta,t_j) \cdot  S_{f_j}(x)
 \\ & \ +O\big( (1+|\beta| x) \text{\rm Err}_J (x,q) \big)  ,
 \end{align*}
 where throughout we have  the error term
\begin{equation}\label{defErrJ}
\text{\rm Err}_J (x,q):= x \frac q{\phi(q)}  (\log\log x)^2 \left(  \frac{ 1}{ (\log x)^{ 1 - \frac{1}{\sqrt{J}}  }} +
  \frac{ 1}{\sqrt{q} (\log x)^{ \eta }} 
 \right)    .
\end{equation}
\end{thm}

It is worth noting that, explicitly,
\begin{equation}\label{defkappaj}
\kappa_j(p^b) := \begin{cases}  f(p^b)/p^{ibt_j}  &\text{ if } p| r_j;\\
\psi_j(p^b)( f_j(p^b) - f_j(p^{b-1})) &\text{ if } p\nmid r_j.
\end{cases}
\end{equation}

The structure of the  main terms in Theorem \ref{thm1}  bears  much in common with those in\eqref{apsumpsi}.
Given $\psi_j,\ t_j$, the only part of the $j$th summand on the right side in Theorem \ref{thm1} involving the values of $f(p)$ with $p\nmid q$ is $S_{f_j}(x)$, which is independent of $\alpha$. 

If  $|\beta|\leq \tfrac 1x \log q\log\log q$ then we can simplify the error term since then
  \[  
(1+|\beta| x)   \text{\rm Err}_J (x,q)\ll x    (\log\log x)^{3+o(1)} \left(  \frac{ 1}{ (\log x)^{ 1 - \frac{1}{\sqrt{J}}  }} +
  \frac{ 1}{\sqrt{q} (\log x)^{ \eta }} 
 \right)    .
\]

For each main term, we have the upper bound
 \begin{align*} 
 &\ll
 \frac{1}{\phi(q)} \cdot 1 \cdot \sqrt{r} \cdot 2^{\omega(q)-\omega(r)}   \cdot \frac 1{\sqrt{1+|\beta| x}} \cdot  \frac{\phi(r)}r x\\
 &= \frac 1{\sqrt{q/r}} \prod_{p|q, p\nmid r} \frac 2{1-1/p} \cdot  \frac {x}{ \sqrt{ q(1+|\beta| x)}}\ll  \frac {x}{ \sqrt{ q(1+|\beta| x)}},
  \end{align*}
 where $\omega(q)$ denotes the number of distinct prime factors of $q$.
which is why we   propose the refined Conjecture in \eqref{1.2b}.
Taking integer $J>1/\varepsilon^2$ we deduce that if \eqref{1.2b} fails then $(\log x)^{ 2 +o(1) }\leq q\leq (\log x)^2(\log\log x)^{1+o(1)}$, as claimed in \eqref{qrange}.

\subsection{Evaluating $R_f(\alpha,x)$ for $\alpha$ rational: Proof of Theorem \ref{thm1} when $\beta=0$} \label{EvalRat}

Writing each integer $n$ as $m\tfrac qd$ where $(m,d)=1$, we have
\[ 
R_f(a/q,x) = \sum_{d|q} f(q/d)  \sum_{\substack{m\leq dx/q \\ (m,d)=1}} f(m)\e\left( \frac{am}d \right) 
= \sum_{d|q} f(q/d)  \sum_{(b,d)=1} \e\left( \frac{ab}d \right) 
 \sum_{\substack{m\leq dx/q \\ m\equiv b (\bmod d)}} f(m).
\]
We evaluate the last sum using   \eqref{apestimate}. As $x>q^2$ the error terms add up to
\[
\ll \sum_{d|q}   \sum_{(b,d)=1}  
\frac {dx}{q\phi(d)} \, \frac{(\log\log x)^2}{(\log x)^{ 1 - \frac{1}{\sqrt{J}}  }} \ll
\frac {q}{\phi(q)} \, \frac{(\log\log x)^2}{(\log x)^{ 1 - \frac{1}{\sqrt{J}}  }} .
\]
 For each fixed $d$ dividing $q$, the $j$th  terms in  \eqref{apestimate} add up to
\[
\frac {S_f(dx/q,\chi_j)}{\phi(d)}      \cdot \sum_{(b,d)=1} \chi_j (b) \e\left( \frac{ab}d \right) =
\cbar_j(a) g(\chi_j) \frac {S_f(dx/q,\chi_j)}{\phi(d)} ,
\]
where $g(\chi_j)=g(\psi_j) \psi_j(\tfrac d{r_j}) \mu(\tfrac d{r_j}) $. Lemma \ref{corComparePC} (with $q=d$ and $\ell=q/d$) implies that  
\[
S_f(dx/q,\chi_j)=   \frac { I(x,0,t_j)}{(q/d)^{1+it_j}}  k_j(d)
S_{f_j}(x)+O \left( \frac{d/r_j}{\phi(d/r_j)}   \, \frac{x(\log\log x)^2 }{ (q/d) (\log x)^{\eta}}  \right).
\]
 Therefore the contribution from $\psi_j$ equals, writing $d=kr_j$, $I(x,0,t_j)S_{f_j}(x)$ times
 \[
 \frac{\pbar_j(a) g(\psi_j) }{\phi(r_j)}  
 \sum_{\substack{k|\tfrac q{r_j}\\ (k,r_j)=1}} \frac{f(q/kr_j) }{(q/kr_j)^{1+it_j}} \frac { \mu(k)\psi_j(k)k_j(k) }{\phi(k)}
 =  \frac{\pbar_j(a) g(\psi_j) }{\phi(q)} \kappa_j\bigg( \frac q{r_j}\bigg)
 \]
 since $k_j(r_j)=1$,  plus an error term of 
 \[
 \ll \frac{ r_j }{\phi(r_j)} 
\frac 1{\sqrt{q/r_j} }    \sum_{\substack{k|\tfrac q{r_j}\\ (k,r_j)=1}}    
\frac{\mu^2(k)  k^2}{\phi(k)^2}   \, \frac{x(\log\log x)^2 }{ \sqrt{q} (\log x)^{\eta}}
\ll \frac{q }{\phi(q)}  \frac{x(\log\log x)^2 }{ \sqrt{q} (\log x)^{\eta}}.
 \]
 This yields Theorem \ref{thm1} when $\beta=0$. 
Moreover we observe that we can choose the same $t_j=t_{f\pbar_j}(x,\log x)$ for any $v\in [x/\log x, x]$.

\hfill \qed
   
 \subsection{Evaluating $R_f(\alpha,x)$ for $\alpha$ irrational: Proof of Theorem \ref{thm1} when $\beta\ne 0$}
 The starting point is    the identity
 \begin{equation} 
 R_f(x,\alpha) = \e(\beta x)  R_f(x,a/q) - 2\pi i \beta
 \int_1^x \e( \beta v)  R_f(v,a/q ) \d v .\label{4.1}
 \end{equation}  
 We truncate the integral at $\tfrac x{\log x}$ at a cost of $\leq 2\pi |\beta|
 \int_0^{x/\log x} v \d v  =\pi |\beta| x\cdot \tfrac x{(\log x)^2}$. We then substitute in the formula for when $\beta=0$ which we established in the precious subsection. Integrating the error term, we obtain the total  error 
 \[
 O\left( (1+|\beta| x)\text{\rm Err}_J (x,q) \right) ,
 \]
with the notation \eqref{defErrJ}.
  The $j$th term becomes
 $\tfrac 1{\phi(q)} \pbar_j(a)   g(\psi_j) \kappa_j(\tfrac q{r_j})I(x,0,t_j)  $ times
 \[
  \e(\beta x)    \sum_{n\leq x} f_j(n) - 2\pi i \beta
\int_{x/\log x}^x \e(\beta v) (v/x)^{it_j} \sum_{n\leq v} f_j(n)\d v .
\]
By   \eqref{adaptTh4},  this has main term
$$
 \left\{ \e(\beta x) - 2\pi i \beta   \int_{x/\log x}^x
\frac {\e(\beta v)}{(x/v)^{1+it_j}} \d v  \right\} \sum_{n\leq x} f_j(n) 
$$
plus the error term
\[
\ll   |\beta|
\int_{x/\log x}^x  v
  \frac{ (\log\log v)^2  }{  (\log v)^{ \eta }}  \d v  \ll  |\beta|x^2 \frac{ (\log\log x)^2  }{  (\log x)^{ \eta }} .
\]
Extending  the  integral  to 0  yields an error term $\ll  |\beta|  \int_0^{x/\log x} \tfrac vx  \d v\cdot  x\leq  |\beta| x\cdot \tfrac x{(\log x)^2}$. Then, integrating by parts, we obtain
$$
\e(\beta x) - 2\pi i \beta  \int_{0}^x \frac {\e(\beta
v)}{(x/v)^{1+it_j}} \d v =  \frac{I(x,\beta,t_j) }{I(x,0,t_j)}.
$$
The result follows by substituting this in above. \qed

 \section{Twisting by periodic functions} \label{sec: twist}
  If   $h $ is a function of period $q$ then for any Dirichlet character $\psi \pmod r$ where $r$ divides $q$, and any integer $D$ for which $r|D|q$ define the \emph{pseudo-Gauss sums}
 \[
 G_h(D;\psi):=\sum_{a=1}^D \psi(a) h\bigg(\frac {aq}D\bigg)
 \]
 We modify the argument of section \ref{EvalRat} to prove the following general result.

 \begin{thm} \label{thm: GenTwist}
 Let $h $ be a function of period $q$ with $q\leq (\log x)^{O(1)}$. Let $f\in \mathcal M$ and define the $\psi_j$ as before. Fix $\varepsilon>0$ and  $J>1/\varepsilon^2$. Then
  \begin{align*} 
\sum_{n\leq x} f(n)h(n)  =& \frac{1}{\phi(q)}  \sum_{\substack{j=1 }}^{J-1}  \bigg( \sum_{ r_j|n|q}  k_j(n) \kappa_j\big(\tfrac qn\big) G_{h}(n;\psi_j) \bigg)   I(x, 0,t_j)  S_{f_j}(x) \\ &  +O\Bigg(  \frac 1q \sum_{m=0}^{q-1} |h(m)|\cdot   \frac{x}{(\log x)^{ 1 - \varepsilon  }} + \frac 1{q}\sum_{\substack{j=1\\ r_j|q}}^{J-1} 
\max_{ r_j|n|q}   |G_{h}(n;\psi_j) |  
\cdot \frac{x  }{  (\log x)^{\eta- \varepsilon}} \Bigg)  .
 \end{align*}
 \end{thm}
 
 The function $h$ enters into the main term in Theorem \ref{thm: GenTwist} only within the pseudo-Gauss sums $G_\psi $.
 In usual Gauss sums one can always reduce to the case where $\psi$ and $h$ have the same period (via some simple identities).
 We cannot do that here so perhaps we should use a different definition for pseudo-Gauss sums ? For example we can rewrite the main term here:
 By expanding the $G_{ \psi_j}(n) $ and re-organizing, the parenthesized part of the $j$th main term can be written as
 \[
\frac 1{q} \sum_{ r_j|m|q} \frac m {\phi(m)} \frac{f(q/m)}{(q/m)^{it}}
\, k_j((m,\tfrac qm)) G_{h}^{\dag}(m;\psi_j) \] 
where \[  G_{h}^{\dag}(m;\psi_j):= \sum_{(b,m)=1} \psi_j(b) h\bigg(\frac {bq}m\bigg).
\]

\begin{proof} We again use our estimates of $f$ in arithmetic progressions:
\[
 \sum_{n\leq x} f(n) h(n)= \sum_{d|q} f\left( \frac{q}d \right)  \sum_{(b,d)=1} h\left( \frac{bq}d \right)
 \sum_{\substack{m\leq dx/q \\ m\equiv b \mod d}} f(m).
\]
We evaluate the last sum using   \eqref{apestimate}. The error terms add up to
\[
\ll \sum_{d|q}   \sum_{(b,d)=1}  \bigg| h\left( \frac{bq}d \right)  \bigg|
\frac {dx}{q\phi(d)} \, \frac{(\log\log x)^2}{(\log x)^{ 1 - \frac{1}{\sqrt{J}}  }} \ll
\frac 1q \sum_{m=0}^{q-1} |h(m)|\cdot 
  \frac{x}{(\log x)^{ 1 - \varepsilon  }} .
\]
taking any integer $J>1/\varepsilon^2$. For each fixed $d$ dividing $q$ the $j$th  terms in  \eqref{apestimate} add up to
\[
\frac {S_f(dx/q,\chi_j)}{\phi(d)}      \cdot \sum_{(b,d)=1} \chi_j (b)  h\left( \frac{bq}d \right)
=\frac {S_f(dx/q,\chi_j)}{\phi(d)}      \sum_{\substack{\ell |d\\ (\ell,r_j)=1}} \mu(\ell) \psi_j(\ell) G_{h}(d/\ell;\psi_j) .
\]
 Lemma \ref{corComparePC} (with $q=d$ and $\ell=q/d$) implies that  
\[
S_f(dx/q,\chi_j)=   \frac { I(x,0,t_j)}{(q/d)^{1+it_j}}  k_j(d)
S_{f_j}(x)+O \left( \frac{d/r_j}{\phi(d/r_j)}   \, \frac{x(\log\log x)^2 }{ (q/d) (\log x)^{\eta}}  \right).
\]
 Therefore the contribution from $\psi_j$ equals, writing $d=kr_j, k=\ell m$ and $n=mr_j$, $I(x,0,t_j)S_{f_j}(x)$ times
 \[
\sum_{ m|\tfrac q{r_j} } G_{h}(mr_j;\psi_j)  \sum_{\ell|\tfrac q{mr_j} } \frac{f(q/\ell mr_j) }{(q/\ell mr_j)^{1+it_j}} \frac {  k_j(\ell m) }{\phi(\ell mr_j)}  \mu(\ell) \psi_j(\ell) =
\frac 1{\phi(q)} \sum_{ r_j|n|q} G_{h}(n;\psi_j)  k_j(n) \kappa_j\bigg(\frac qn\bigg)  
\]
We also have an error term of 
 \begin{align*}
& \ll \frac 1{q} \sum_{ r_j|n|q}|G_{h}(n;\psi_j)|\sum_{\substack{\ell| q/n\\ (\ell, r_j)=1} }   \mu^2(\ell) \frac{\ell n}{\phi(\ell n)} \frac{\ell n/r_j}{\phi(\ell n/r_j)} 
  \, \frac{x(\log\log x)^2 }{  (\log x)^{\eta}} \\
 & \ll \frac 1{q} \sum_{ r_j|n|q} |G_{h}(n;\psi_j)| 2^{\omega(q/n)}   \, \frac{x(\log\log x)^{2+o(1)} }{  (\log x)^{\eta}}\ll
\max_{ r_j|n|q}   |G_{h}(n;\psi_j)|  
\cdot \frac{x  }{  q(\log x)^{\eta+o(1)}} 
  \end{align*}
  Collecting together these estimates gives the Theorem.
 \end{proof}

\subsection{Periodic functions of size one} \label{PerSizeOne} 

In this subsection we will restrict our attention to periodic functions $h$, of minimal period $q$, with the property that 
\[
h(n) = \prod_{p^e\| q} h_p(n)
\]
where each $h_p$ has minimal period $p^e$. Examples include  characters mod $q$ and exponentials like
$\e(\tfrac{g(n)}q)$ where $g(x)\in \mathbb Z[x]$ or $g(n)=an+b\overline{n}$ defined only if $(n,q)=1$, where 
$\overline{n}$ is the inverse of $n\pmod q$. Moreover $h(.)$ might be the product of such functions, like
\[
h(n):= \chi_1(n+a_1)\cdots \chi_m(n+a_m)\e(\tfrac{g(n)}q)
\]
where we are given characters $\chi_1,\ldots,\chi_m$  $\pmod q$ and    integers $a_1,\ldots,a_m$ for some $m\geq 1$, with $g(\cdot)$ as above.  For each of these cases, Theorem \ref{thm2} of Chapter 6 in  \cite{WLi}   gives that
\[
\bigg| \sum_{n \mod {p^e}} h_p(n)  \bigg| \leq (m+d)\,   p^{e/2}
\]
for all prime powers $p^e$, and so by the the Chinese Remainder Theorem
\[
\bigg| \sum_{n \mod q} h(n) \bigg| \leq (m+d)^{\omega(q)}  q^{1/2}.
\]
We can also apply this result for $\psi h $, where $\psi$ has conductor dividing $q$, provided this also has minimal period $q$.

We wish to apply  Theorem \ref{thm: GenTwist} to these $h $. We restrict attention to $q$ squarefree
(to avoid the case where $p$ divides both $r$ and $q/D$). In this case we can deduce the bounds
$|G_h(D;\psi)|\leq (m+d)^{\omega(D)}  D^{1/2}$; and the error term in Theorem \ref{thm: GenTwist} becomes
\[
 \ll    \frac{x}{(\log x)^{ 1 - \varepsilon  }} +   \frac{x  }{ \sqrt{q} (\log x)^{\eta- \varepsilon}} .
\]
 Let $h_m(n) = \prod_{p^e\| m} h_p(n)$ whenever $m|q$. If $(D,q/D)=(r,D/r)=1$ (which always happens  if $q$ is squarefree) then by the Chinese Remainder Theorem we have 
 \[
 G_h(D;\psi) = h_{q/D}(0) \pbar(\tfrac qD) \cdot \prod_{p^e\|r} \sum_{a=1}^{p^e} \psi_p(a) h(a) \cdot
  \prod_{p^f\|D/r} \sum_{b=1}^{p^f}   h_p(b)
 \]
where $\psi=\prod_{p^e\|r} \psi_p$ and $\psi_p$ is a character mod $p^e$. Therefore the coefficients in Theorem~\ref{thm: GenTwist} become, if $(r,q/r)=1$ and $q/r$ is squarefree, writing $f^*(p)$ in place of $f_j(p)$ (and suppressing the subscript $j$ everywhere)
 \[
  \sum_{ r|n|q}  k(n) \kappa\big(\tfrac qn\big) G_{h}(n;\psi) =  \prod_{p^e\|r}   \sum_{a=1}^{p^e} \psi_p(a) h(a) \cdot
   \prod_{p| \tfrac qr}\bigg( (1-\tfrac 1p) h_p(0) f^*(p) + (1-\tfrac { f^*(p)}p) \sum_{a=1}^{p-1}   h(a) \bigg)
  \]
Taking absolute values, in the above examples, its absolute value  is 
\[
\leq \prod_{p^e\|r}  (m+d)\,   p^{e/2}     \prod_{p| \tfrac qr} (p-1)(1+\tfrac 2p)  \leq (m+d)^{\omega(r)} \frac q{\phi(q)} \frac q{\sqrt{r}}
\]
Therefore we deduce that 
 \begin{equation}  \label{eq: thmchisums}
\sum_{n\leq x} f(n)h(n)  =  \sum_{\substack{j=1\\ r_j|q}}^{J-1}   \frac{c_{j,q}}{\sqrt{r_j}}    \frac{x^{it_j}}{1+it_j}  S_{f_j}(x)   +O\Bigg(  \frac{x}{(\log x)^{ 1 - \varepsilon  }} +   \frac{x  }{ \sqrt{q} (\log x)^{\eta- \varepsilon}} \Bigg)  ,
 \end{equation}
 where each $|c_{j,q}| \leq (m+d)^{\omega(r_j)} \frac {q^2}{\phi(q)^2}$.

\section{Development of the circle method}

We will apply Theorem \ref{thm1} to applications of  the circle method using multiplicative functions.

\subsection{Major and minor arcs} 

Let $Q=x/(\log x)^{\tau-\varepsilon}$ where $\tau=\tfrac{2-\sqrt{2}}3$, and for each $\alpha$ on the unit circle we select $q\leq Q$ such that $|\alpha -a/q|\leq 1/qQ$ for some integer $a$ coprime with $q$. Define the major arcs ${\mathfrak M}$   by
\[
{\mathfrak M}=\bigcup_{\colt{(a,q)=1}{q\leq x/Q}} \left[\frac aq - \frac{1}{qQ},\ \frac aq + \frac{1}{qQ} \right],
\] 
and let ${\mathfrak m}$ the minor arcs defined by $[0,1)\smallsetminus {\mathfrak M}$, where $[0,1)$ stands for $\mathbb R/\mathbb Z$.

\begin{lem} \label{ArcEstimates}
With the notation as above (but with $\psi=\psi_1, \kappa=\kappa_1, r=r_1, t=t_1, f_*=f_1$ we can write
$$R_f(\alpha,x) =M_f(\alpha,x ) +E_f(\alpha,x )$$  
where $M_f(\alpha)=M_f(\alpha,x)=0$ if $ \alpha\in {\mathfrak m}$ and 
\[
M_f(\alpha)= \frac{ \pbar(a)  g(\psi)    1_{r|q} \, \kappa(\tfrac qr)}{\phi(q)}    I(x, \beta,t) \cdot  S_{f_*}(x)   \qquad\text{ if } \alpha\in {\mathfrak M},
\]
with 
\begin{equation} \label{E-bounds}
 E_f(\alpha)=E_f(\alpha,x) \ll    \frac {x}{ (\log x)^{\tau /2+o(1)} } \qquad\text{ for all } \alpha.
\end{equation}
We also have that 
\[
\int_0^1 |R_f(\alpha,x)|^2\d \alpha, \ \int_{0}^1 |M_f(\alpha)|^2 \d \alpha, \ \int_{0}^1 |E_f(\alpha,x)|^2 \d \alpha \ll  x.
\]
\end{lem}

\begin{proof} We get better estimates for the minor arcs than claimed here, from \eqref{1.1}.
Taking $J=2$ in Theorem  \ref{thm1}, we obtain the main term $M_f(\alpha,x )$, with 
\[
E_f(\alpha,x ) \ll \bigg(1+\frac x{qQ}\bigg)\frac{ x}{ (\log x)^{ \tfrac 32 \tau  +o(1) }}\ll
 \frac {x}{ (\log x)^{\tfrac \tau2+\varepsilon+o(1)} }  
\]
Moreover
\begin{align*}
\int_{\mathfrak M} |E_f(\alpha,x)|^2  = &\sum_{q\leq x/Q} \sum_{(a,q)=1} \int_{-1/qQ}^{1/qQ} |E_f(\alpha+a/q,x)|^2\d \alpha\cr& \ll
\sum_{q\leq x/Q} \frac{\phi(q)}{qQ}\, \frac {x^2}{q^2Q^2} \,   \frac{ x^2 }{  (\log x)^{ 2-\sqrt{2}+o(1) }} 
\\
&\ll    \frac{ x^4 }{ Q^3 (\log x)^{3\tau+o(1) }} \ll \frac x{(\log x)^{3\varepsilon+o(1)} } =o(x).
\end{align*}

  Now, if  $\alpha\in {\mathfrak m}$ then $x/Q\leq q\leq Q$ and so we may take $R=x/Q$ in \eqref{1.1} to obtain
\[
M_f(\alpha)=0 \text{ and } E_f(\alpha,x) \ll    \frac {x}{ (\log x)^{\tfrac 12(\tau-\varepsilon)+o(1)} }  \text{ if } \alpha\in {\mathfrak m}.
\]
Letting $\varepsilon\to 0$ we  deduce \eqref{E-bounds}. By Parseval we have
\[
\int_0^1 |R_f(\alpha,x)|^2\d \alpha =\sum_{n\leq x} |f(n)|^2\leq x,
\] 
and so
\[
\int_{0}^1 |M_f(\alpha)|^2 \d \alpha=
\int_{{\mathfrak M}} |M_f(\alpha)|^2 \d \alpha \leq 2 \int_{{\mathfrak M}}
\big(|R_f(\alpha,x)|^2+|E_f(\alpha,x)|^2\big) \d \alpha  \ll x, 
\]
and therefore
\[
\int_{0}^1 |E_f(\alpha,x)|^2 \d \alpha
=\int_{{\mathfrak M}} |E_f(\alpha,x)|^2 \d \alpha +
\int_{{\mathfrak m}} |R_f(\alpha,x)|^2 \d \alpha \ll x ,
\]
as desired.
\end{proof}

\subsection{Mean value of a multiplicative function on a weighted set of integers}  
Given $f\in \mathcal M$ define, for convenience, $S_f(x)=: x\mu(f,x)$ where $\mu(f,x):=\tfrac 1x\sum_{n\leq x} f(n)$.
The ``structure theorem'' of \cite{GSBook} states that for a given $f\in \mathcal M$, 
\begin{equation} \label{struceq}
\mu(f,x) = \mu(f^{(s)},x)\mu(f^{(\ell)},x) 
+ O\left(    \frac{(\log\log x)^{1+2\eta}}{(\log x)^{\eta}}\right) .
\end{equation}
where  $f^{(s)}, f^{(\ell)}$ are  multiplicative functions with $f=f^{(s)}f^{(\ell)}$ defined by
\begin{equation} \label{deffsfl}
 f^{(s)}(p)=\begin{cases} f(p)p^{-it}  & \text{if} \ p\leq z \\
       1&  \text{if} \ p>z 
      \end{cases}  \ \ \text{and} \ \ 
f^{(\ell)}(p)=\begin{cases} p^{it} & \text{if} \ p\leq z \\
       f(p)&  \text{if} \ p>z
      \end{cases} ,
\end{equation}
where $t=t_f(x,\log x)$ and $z=(\log x)^A$ for some constant $A>0$.
Here we similarly define the two multiplicative functions $F_s, F_\ell \in \mathcal M$ with $F_sF_\ell=f$ as follows:
\[
F_s(p)=\begin{cases}  f(p)& \text{ for } p\leq z ,\\
\psi(p)p^{it} & \text{ for }  p>z, \end{cases}\ \ \text{and} \ \ 
F_\ell(p)=\begin{cases} 1& \text{ for } p\leq z ,\\
f_*(p):=f(p) \pbar(p)p^{-it} & \text{ for }  p>z. \end{cases}
\]
 We note that $(f_*)^{(s)}=(F_s)_*$ and $(f_*)^{(\ell)}=F_\ell=(F_\ell)_*$.

We now use Theorem \ref{thm1} to obtain an estimate for the sum over $n\leq x$, of $f(n)$ times an arbitrary weight $w_n$. 

\begin{prop} \label{wtProp}  If $f\in \mathcal M$ and $\{ w_n\}_{ n\leq x}$ is a set of weights then
\begin{equation}\label{sieveDecomp}
 \sum_{n\leq x} w_n f(n) =  \frac 1x \sum_{n\leq x} F_\ell(n) \cdot  \sum_{n\leq x} w_n F_s(n) +O\Big(\|W\|_1  \frac {x}{ (\log x)^{\tau /2+o(1)} } \Big)
\end{equation}
where $g=f_1$, $W(t)=\sum_{n\leq x} w_n \e(-nt)$ and  $\|W\|_1:=\int_0^1  |W(\alpha)| \d \alpha$.
\end{prop}

\begin{proof} By Plancherel's theorem we have 
$$
\sum_{n\leq x} w_n f(n) =  \int_0^1  R_f(\alpha,x) W(\alpha) \d \alpha
$$
and, by \eqref{E-bounds}, this is 
\[
 \int_0^1  M_f(\alpha) W(\alpha) \d \alpha + O\left(  \|W\|_1  \frac {x}{ (\log x)^{\tau /2+o(1)} } \right).
\]
Now $M_f(\alpha)=0$ if $\alpha\in {\mathfrak m}$, and so by Lemma \ref{ArcEstimates}
$$
 \int_0^1  M_f(\alpha) W(\alpha) \d \alpha =  \sigma(f) S_{f_*}(x)  ,
$$
 where $S_{f_*}(x):=\sum_{n\leq x} f_*(n)$ and
\[
\sigma(f)= \sum_{\substack{ q\leq x/Q \\ r|q}} \sum_{(a,q)=1}
 \frac{ \pbar(a)  g(\psi)   \kappa(\tfrac qr)}{\phi(q)}   
\int_{-1/qQ}^{1/qQ}     I(x, \beta,t)     W(a/q+\beta) \d\beta .
\]
Therefore 
\[
\sum_{n\leq x} w_n f(n) = \sigma(f) S_{f_*}(x)   +O\left(  \|W\|_1  \frac {x}{ (\log x)^{\tau /2+o(1)} } \right).
\]
Now \eqref{struceq} and the discussion that follows it, implies that 
\[
S_{f_*}(x) =  \mu(f_*,x) x =  \mu((F_s)_*,x)\mu(F_\ell,x)  x + O\left(    \frac{x}{(\log x)^{\eta+o(1)}}\right) .
\]
 Moreover  the only term  in the summands for $\sigma(f)$  that directly involves values of $f$ is
$\kappa(q/r)$, and this involves only  $f(p)$ for primes $p\leq x/Q$. Now $f(p)=F_s(p)$ for all $p\leq x/Q$ (as $x/Q\leq z$) and, by definition, $\psi_f=\psi_{F_s}$ and $t_f(x,\log x)=t_{F_s}(x,\log x)$.
Therefore 
\[ \sigma(f)=\sigma(F_s), \]
and so 
\[
\sum_{n\leq x} w_n f(n) = \sigma(F_s) \mu((F_s)_*,x)\mu(F_\ell,x)  x +O\left(  \|W\|_1  \frac {x}{ (\log x)^{\tau /2+o(1)} }  \right),
\]
since $|\sigma(f)|\ll  \|W\|_1 $ by definition and $\tau/2<\eta$.
By the same argument we also have 
\[
\sum_n w_n F_s(n)= \sigma(F_s) \mu((F_s)_*,x) x  + O\left(  \|W\|_1  \frac {x}{ (\log x)^{\tau /2+o(1)} }  \right) ,
\]
and   \eqref{sieveDecomp} follows by comparing the last two displayed equations, as
$|\mu(F_\ell,x)|\leq 1$.
\end{proof}

\subsection{Applications}  

\begin{thm}\label{thm2} Let $f\in \mathcal M$ and $z=\log x$.   If $A, B\subset \{ 1,2,\ldots, [x]\}$ with $A+ B\subset \{ 1,2,\ldots, x\}$, then
\begin{align*}
\frac 1 {|A||B|}  \sum_{\colt{a\in A}{b\in B}} f(a+b)   = 
\frac 1x \sum_{n\leq x} F_\ell (n) \frac 1 {|A||B|}  \sum_{\colt{a\in A}{b\in B}} F_s(a+b)
+ O\left(\frac {x }{(|A||B|)^{1/2} (\log x)^{\tau /2+o(1)} } \right) .
\end{align*}
\end{thm}

The effect of the large primes is independent of the particular choice of sets $A$ and $B$. The form of the error term is classic in this context.

 \begin{proof} Let  \[w_n=\frac{1}{|A||B|}\#\{ (a,b)\in A\times  B\,:\, n=a+b\},\] and
proceed as above by noting that
$W(\alpha)=A(\alpha)B(\alpha)/|A||B|$ so that
\begin{align*}\|W\|_1&=
\int_0^1  |W(\alpha)| \d \alpha =
\frac 1{|A||B|} \int_0^1  |A(\alpha)B(\alpha)| \d \alpha 
\\&\leq \frac 1{|A||B|} \left(  \int_0^1  |A(\alpha)|^2 \d \alpha   \int_0^1  |B(\alpha)|^2 \d \alpha  \right)^{1/2} = \frac 1{(|A||B|)^{1/2}} ,
\end{align*}
which, together with \eqref{sieveDecomp},  implies the result.
\end{proof}

\subsection{Explicit Theorem \ref{thm2}:  The mean value of $F_s(a+b)$ over $a\in A,\ b\in B$} 

 \begin{lem}\label{lemcor1}  With the notation as in  Theorem \ref{thm2}, we have
\[
 \frac 1 {|A||B|}  \sum_{\colt{a\in A}{b\in B}} F_s(a+b)  = \sum_{P(m)\leq z}  \kappa(m)  \cdot \frac 1 {|A||B|}
 \sum_{\tolt{a\in A}{b\in B}{ m| a+b}}  \psi\left( \frac{a+b}m\right)   (a+b)^{it}
\]
\end{lem}

\begin{proof}  If $\chi_z$ the characteristic function of $z$-friable
integers then we have identity
$$F_s(n)=n^{it}( (\kappa\chi_z)*\psi)(n).$$
We deduce that
\begin{equation}
 \sum_{n } w_n F_s(n) = \sum_{P(m)\leq z } \kappa(m)
 \sum_{{n: \ m|n }} w_n n^{it}  \psi(n/m) , \label{5.2}
\end{equation}
which gives   the result.
\end{proof}

It is also worth observing that by inclusion-exclusion we have 
\[
\frac 1x \sum_{n\leq x} F_\ell(n) =  \prod_{p\leq z}  \frac{(1-f_*(p)/p)}{(1-1/p) } \cdot \frac 1x \sum_{n\leq x} f_*(n)   + O\left(    \frac{1}{(\log x)}\right) .
\]
 Together with Lemma \ref{lemcor1} this allows us to give a more explicit main term in Theorem \ref{thm2}.

\section{Three term products in the circle method}

 \subsection{Proof sketch of the formulas of Theorems \ref{thm3} and \ref{thm4}}
We have
\[
\begin{split}
 \sum_{\colt{\ell ,m ,n \leq x}{a \ell +b m +c n =0 }}   f(\ell ) g(m )h(n )  = &\sum_{\ell ,m ,n \leq x }  f(\ell )g(m )h(n )
\int_0^1 \e(t(a \ell +b m +c n )) \d t
\\  =& \int_0^1 R_f(a t,x) R_g(b t,x) R_h(c t,x) \d t,
\end{split}
\]
and, similarly,
\begin{equation}  \sum_{\colt{\ell,m,n\geq 1}{ \ell+m+n=N }}
 f(\ell) g(m)h(n)  = \int_0^1 \e(-N t) R_f(t) R_g(t) R_h(t) \d t   \label{6.4}
\end{equation}
We decompose the integrands using the identity
\begin{align*}
R_f(\alpha,x)& R_g(\beta,x) R_h(\gamma,x) -M_f(\alpha) M_g(\beta) M_h(\gamma)\\
&= R_f(\alpha,x) R_g(\beta,x) E_h(\gamma)  + R_f(\alpha,x) E_g(\beta) M_h(\gamma) + E_f(\alpha) M_g(\beta) M_h(\gamma) .
\end{align*}
For any functions $A(\alpha)$ and  $B(\alpha)$ we have
\begin{align*}
\left| \int_0^1 A(\alpha)B(\alpha)E(\alpha) \d \alpha \right| 
&\leq  \max_\alpha |E(\alpha) | \int_0^1 |A(\alpha)B(\alpha)| \d \alpha \\
&\leq  \max_\alpha |E(\alpha) |   \int_0^1 \big(|A(\alpha)|^2 + |B(\alpha)|^2\big) \d \alpha   
\end{align*}
Therefore by the estimates of   Lemma \ref{ArcEstimates} we deduce that each of
\[
  \int_0^1 R_f(a t,x)  R_g(b t,x) E_h(c t) \d t , \ \int_0^1 R_f(a t,x)  E_g(b t) M_h(c t) \d t\]  \text{and } \[\int_0^1 E_f(a t)  M_g(b t) M_h(c t) \d t 
  \]
are $\ll  {x^2}/{ (\log x)^{\tau /2+o(1)} }$. Therefore
\[
 \sum_{\colt{\ell ,m ,n \leq x}{a \ell +b m +c n =0 }}   f(\ell ) g(m )h(n )   =\int_0^1 M_f(a t) M_g(b t) M_h(c t) \d t + O\left(    \frac {x^2}{ (\log x)^{\tau /2+o(1)} }   \right) , 
\]
and
\[
\sum_{\colt{\ell,m,n\geq 1}{ \ell+m+n=N }}
 f(\ell) g(m)h(n)  = \int_0^1 \e(-N t) M_f(t) M_g(t) M_h(t) \d t + O\left(    \frac {x^2}{ (\log x)^{\tau /2+o(1)} }   \right).
\]

From here we imitate the proof of Proposition \ref{wtProp}: First we replace $M_f,M_g$ and $M_h$ by the relevant expressions given in Lemma \ref{ArcEstimates} and the main term becomes an integral over the major arcs.
The  main term  is the product of the mean values of 
$f_*,g_*$ and $h_*$, times an expression that  depends only  on the values of $f(p), g(p)$ and $h(p)$ with $p\leq z$. Therefore if we compare this with the same calculation, but now working with $F_s, G_s, H_s$ we discover that the (complicated) main term is the same in each case, other than now we have the 
product of the mean values of $(F_s)_*,(G_s)_*$ and $(H_s)_*$. We can compensate for this by multiplying through by the product of the mean values of $(F_\ell)_*=F_\ell,(G_\ell)_*=G_\ell$ and $(H_\ell)_*=H_\ell$.
We deduce the formulas in the first parts of Theorems  \ref{thm3} and \ref{thm4}.
\qed

  Matthiesen \cite{Ma20}  proved weaker but more general results of this type using very different methods.

\section{Explicit evaluation of the mean value of $F_s(\ell)G_s(m)H_s(n)$}

In this section we complete  the proofs of Theorem \ref{thm3} and \ref{thm4}.

\subsection{The mean value of $F_s(\ell)G_s(m)H_s(n)$ over solutions to $a\ell+bm+cn=0$.}    \label{EvalCsts}
In this subsection we will sketch a proof that 
\begin{equation}  \label{FsGsHssum}
\frac 1{x^2/2} \sum_{\colt{\ell ,m ,n \leq x}{ a \ell +b m +c n =0 }}   F_s(\ell) G_s(m)H_s(n) 
= 2{\cal E}(\infty)  x^{it} \ \delta_{f,g,h} \ \prod_{p\leq z} {\cal E}(p) + O\bigg( \frac 1{\log x} \bigg)
\end{equation} 
where $t=t_f+t_g+t_h$,  the factor $\delta_{f,g,h}=1$ if $\psi_f\psi_g\psi_h$ is principal, and $0$ otherwise,
$$
{\cal E}(\infty)=  \frac 1{|c|}  \int_{\colt{0\leq u,v,w \leq 1}{au+bv+cw=0}}  u^{it_f}
v^{it_g}  w^{it_h}\ \d u \d v,
$$
 the  ${\cal E}(p)$ are  appropriate local factors for each
prime $p\leq z$.

The main idea is to replace each $F_s(\ell)$ (as well as $G_s(m)$ and $H_s(n)$) by a periodic function times
$\ell^{it_f}$, and so evaluate them separately. Let $e_p$ be the smallest integer for which $p^{e_p}>z^2$, and
$N=r_fr_gr_h \prod_{p\leq z} p^{e_p}$.

If $v$ is the largest $z$-friable divisor of $n$ then $F_s(n)=f_*(v)\psi(n)n^{it}$.
Define 
$f_\dag(p^k):=f_*(p^{\min\{ k,e_p\} })\psi(p^k)$ for all $p\leq z$, and $f_\dag(p^k):=\psi(p^k)$ if $p>z$,
so that $F_s(n)=f_\dag(n)n^{it}$ provided $v$ divides $N$. The number of $n\leq x$ for which $F_s(n)\ne f_\dag(n)n^{it}$ is
\[
\leq \sum_{p\leq z} \frac x{p^{e_p}} < \frac{x\pi(z)}{z^2} <\frac xz.
\]
Therefore the accumulated error term in our sum by this change is $\ll x^2/z$ which is acceptable.

Now if $u\equiv n \pmod N$ then $\psi(n)=\psi(u)$ as $r_\psi$ divides $N$ and if $f_\dag=f_\dag^*\psi_f$ then
$f_\dag^*$ only depends on the primes dividing $N$ and so
$f_\dag^*(n)=f_\dag^*((n,N))=f_\dag^*((u,N))=f_\dag^*(u)$. Therefore each $f_\dag(n)= f_\dag(u)$ and so the sum in \eqref{FsGsHssum} is equal to
\[
 \sum_{\substack{u,v,w \mod N  \\  a u +b v +c w \equiv 0 \mod N}}  f_\dag(u) g_\dag(v) h_\dag(w)
  \sum_{\substack{\ell ,m ,n \leq x  \\  a \ell +b m +c n =0   \\  \ell \equiv u,\ m\equiv v,\ n\equiv w \mod N}}  
  \ell^{it_f} m^{it_g} n^{it_h} +O\bigg( \frac {x^2}{\log x} \bigg).
\]
Since $a,b,c$ are pairwise coprime, we see that the inner sums are all non-empty.
Given one term $\ell_0,m_0,n_0$ in the inner sum, the others all look like
$\ell_0+\alpha N, m_0+\beta N, n_0+\gamma N$ where $a\alpha+b\beta+c\gamma=0$.
We deduce that the outer sum is $(x/N)^2 {\cal E}(\infty)  x^{it}(1+O(1/\log x))$.
 Therefore the left-hand side of \eqref{FsGsHssum}  equals
 \[
2{\cal E}(\infty)  x^{it} \cdot \frac 1{N^2} \sum_{\substack{u,v,w \mod N  \\  a u +b v +c w \equiv 0 \mod N}}  f_\dag(u) g_\dag(v) h_\dag(w)   +O\bigg( \frac 1{\log x} \bigg).
\]
The sum remains unchanged if we multiply $u,v$ and $w$ through by any reduced residue $t\pmod N$.  This changes each summand by a factor $(\psi_f\psi_g\psi_h)(t)$. This implies that either the sum is $0$, or 
$(\psi_f\psi_g\psi_h)(t)=1$ for all such $t$, in which case
$\psi_f\psi_g\psi_h$ is principal.

The inner sum can be made into an Euler product by the Chinese Remainder Theorem, and therefore we obtain
\eqref{FsGsHssum}. To be more precise about the Euler factors ${\cal E}(p)$ in the case that $p\nmid abcr_fr_gr_h$:
\begin{align*}
{\cal E}(p) &= \frac 1{p^{2e}} \sum_{\substack{u,v,w \pmod {p^{e}}  \\  a u +b v +c w \equiv 0 \pmod {p^{e}} }}  f_\dag((u,p^e)) g_\dag((v,p^e)) h_\dag((w,p^e)) \\
&= \frac 1{p^{2}}  (p^2-3p+2 +(p-1) ( f_\dag(p) +g_\dag(p) +h_\dag(p) )+O(1))\\
&= 1 - \frac{1-f_\dag(p)}p - \frac{1-g_\dag(p)}p- \frac{1-h_\dag(p)}p+ O\bigg( \frac 1 {p^2} \bigg) 
\end{align*}
since the probability that $p^2$ divides one of $u,v,w$, or $p$ divides two of $u,v,w$ is $O(1/p^2)$.
We deduce that 
\[
\prod_{p\leq z} {\cal E}(p) \asymp \exp\left( -\sum_{p\leq z} \frac{1-f_\dag(p)}p - \sum_{p\leq z} \frac{1-g_\dag(p)}p- \sum_{p\leq z} \frac{1-h_\dag(p)}p \right) .
\]
Now letting $z\to \infty$ we see that if this converges then $\psi_f,\psi_g,\psi_h$ are each principal, and so $\psi_f=\psi_g=\psi_h=1$ since they are all primitive characters.
This completes  the proof of Theorem \ref{thm3}.

The exactly analogous argument completes  the proof of Theorem   \ref{thm4}.

\subsection{When $f,g$ and $h$ are real-valued in Theorems \ref{thm3} and \ref{thm4} } \label{Exp4}

If we have a non-zero mean value in Theorem \ref{thm3} or   \ref{thm4} then $\psi_f=\psi_g=\psi_h=1$ so that 
if $p\leq z$ then $f_\dag(p)=f_*(p)=f(p)p^{-it_f}$. Therefore if $f$ is real-valued and $\prod_{p\leq z} {\cal E}(p) \gg 1$ then $\sum_{p\leq z} \tfrac{1-f_\dag(p)}p\ll 1$ and so $t_f=0$ (and similarly $t_g=t_h=0$).
Thus  
\[
{\cal E}(\infty)  x^{it} =  \frac 1{|c|}  \int_{\colt{0\leq u,v,w \leq 1}{au+bv+cw=0}}    \d u\, \d v,
\]
 in Theorem \ref{thm3} (which equals $\tfrac 12$ if $a=b=-c=1$),
and $2{\cal E}(\infty)  x^{it} =  1$ in Theorem \ref{thm4}.

When $a=b=1, c=-1$ the Euler product term is
\begin{align*}
{\cal E}(p)&=\lim_{e\to \infty} \frac 1{p^{2e}} \sum_{u+v\equiv w \mod {p^e}}  f((u,p^e)) g((v,p^e)) h((w,p^e)) \\
&=:{\cal E}^*(p)\cdot  \left( 1 -\frac{1}{p} \right)^3  \Big( 1 - \frac{f (p)}p  \Big)^{-1} \Big( 1 - \frac{g (p)}p  \Big)^{-1} \cdot 
   \Big( 1 - \frac{h(p)}p  \Big)^{-1},
\end{align*}
 say. Now $F_s=f^{(s)}$ and  $F_\ell=f^{(\ell)}$, and 
\[
\sum_{n\leq x} f_s(n) \sim \prod_{p\leq z}   \left( 1 -\frac{1}{p} \right)  \Big( 1 - \frac{f (p)}p  \Big)^{-1} \cdot x,
\]
by inclusion-exclusion, so that by using \eqref{struceq} we have 
\[
 \prod_{p\leq z}   \left( 1 -\frac{1}{p} \right)  \Big( 1 - \frac{f (p)}p  \Big)^{-1} \sum_{n\leq x} F_\ell(n) \sim
 \sum_{n\leq x} f(n) .
\]
 Substituting this all back into Theorem \ref{thm3} we obtain
 \begin{align}\label{NeedLater}
  \frac 1{x^2/2} \sum_{\colt{\ell,m,n\leq x}{ \ell+ m = n}}  f(\ell)g(m)h(n)= 
 \mu(f,x) \mu(g,x) \mu(h,x)    \prod_{\substack{p\leq z }}  {\cal E}^*(p)
+o(1)     . 
\end{align} 

If say $h(p)=1$ then
\begin{align*}
{\cal E}(p)&=\lim_{e\to \infty} \frac 1{p^{e}} \sum_{u  \pmod {p^e}}  f((u,p^e)) \cdot \frac 1{p^{e}} \sum_{u  \pmod {p^e}} g((v,p^e)) \\ & =  \left( 1 -\frac{1}{p} \right)^2  \Big( 1 - \frac{f (p)}p  \Big)^{-1}   \Big( 1 - \frac{g (p)}p  \Big)^{-1},
\end{align*}
and so ${\cal E}^*(p)=1$. Therefore if ${\cal E}^*(p)\ne 1$ then $f(p),g(p),h(p)\ne 1$.


If  $ f(p)=g(p)=h(p)=-1$ then 
 \[
{\cal E}^*(p):=\bigg( 1 - \frac{  8}{(p-1)^2} \Big( 1 +\frac{   1 }{p^2} \Big)^{-1}   \bigg) .
\]
Therefore, when $f,$ $g,$ $h$ take only $\in\{-1,1\}$, the right-hand side of \eqref{NeedLater} becomes
\[
\mu(f,x) \mu(g,x) \mu(h,x)    \prod_{\substack{p\leq z\\ f(p)=g(p)=h(p)=-1}}  \Big( 1 - \frac{  8p^2}{(p-1)^2(p^2+1)} \Big)
+o(1)     . 
\]
\noindent{\it Remark. }  This is $\sim \mu(f,x) \mu(g,x) \mu(h,x) $    if and only if   $ \mu(f,x) \mu(g,x) \mu(h,x)\gg 1$ and $P:=\{ p\leq z:\  f(p)=g(p)=h(p)=-1\} = \emptyset$: To see this suppose $P\ne \emptyset$.
If odd $p\in P$ then $|{\cal E}^*(p)|<1$ and so $\prod_{p\in P}|{\cal E}^*(p)| <1$ unless $2\in P$. But then
$\prod_{p\in P}|{\cal E}^*(p)| \geq   \prod_{p } |1 - \tfrac{  8p^2}{(p-1)^2(p^2+1)} | = 1.322\ldots>1$, 
\smallskip

If $f,g,h$ only take values $0$ or $1$ (with $a=b=1, c=-1$) then
 \begin{align*}
 \frac 1{x^2/2} \sum_{\colt{\ell,m,n\leq x}{ \ell+ m = n}}  f(\ell)g(m)h(n)= 
  \mu(f,x) \mu(g,x) \mu(h,x)     \prod_{\substack{p\leq z\\ f(p)=g(p)=h(p)=0}}  \Big( 1 - \frac{  1}{(p-1)^2 } \Big) 
+o(1),
\end{align*} 
from which we immediately deduce \eqref{ABC1}.

We can proceed analogously in  Theorem \ref{thm4}, obtaining
\[
  \frac 1{x^2/2} \sum_{\colt{\ell,m,n\leq x}{ \ell+ m + n=N}}  f(\ell)g(m)h(n)= 
 \mu(f,x) \mu(g,x) \mu(h,x)    \prod_{\substack{p\leq z }}  {\cal E}^*_N(p),
+o(1)     
\]
for appropriate Euler factors ${\cal E}^*_N(p)$, and again we note that ${\cal E}^*_N(p)=1$
if $f(p),g(p)$ or $h(p)=1$. If $f(p)=g(p)=h(p)=0$ then 
\[
{\cal E}^*_N(p)= \left( 1 -\frac{1}{p} \right)^{-3} \frac 1{p^2} \sum_{\substack{ u+v+w\equiv N \pmod {p} \\ (uvw,p)=1}}  1 =\begin{cases}
1+\tfrac 1{(p-1)^3} & \text{ if } p\nmid N;\\
1-\tfrac 1{(p-1)^2} & \text{ if } p| N,
\end{cases}
\]
which implies \eqref{ABC2}. If $f(p)=g(p)=h(p)=-1$ then one can similarly calculate a formula for ${\cal E}^*_N(p)$
but it appears to be complicated.

\subsection{Sketch of proof of Corollary \ref{cor2}}  The density of solutions to
$f(a)=g(b)=h(c)=-1$ with $a+b=c\leq x$ is  
\begin{equation}
\frac 1{x^2/2} \sum_{\colt{a,b,c\leq x}{a+b=c }} \tfrac 18           
(1-f(a))(1-g(b))(1-h(c)). \label{7.2}
\end{equation}
When we expand this we have the sum of eight mean values and we can apply \eqref{NeedLater} to them all,
noting that ${\cal E}^*(p)=1$ except in the term with the product $f(a) g(b) h(c)$. Therefore   if the mean values of $f,g$ and $h$ are $\delta_f,\delta_g$ and $\delta_h$, respectively, then  \eqref{7.2} is
\begin{equation}
  \tfrac 18 \left(    (1-\delta_f)  (1-\delta_g) (1-\delta_h)+ (1-C_{\cal P} )\delta_f\delta_g\delta_h\right) +o(1) \, \label{7.4}
\end{equation}
where
\[
 C_{\cal P}:= \prod_{p\in \cal P}\Big( 1 - \frac{  8p^2}{(p-1)^2(p^2+1)} \Big)    
 \]
 with $\cal P:=\{ p: f(p)=g(p)=h(p)=-1\}$, by \eqref{NeedLater}. We need to maximize  the expression in \eqref{7.4}: 

By   the main result of  \cite{GS01}, we have $\delta_f,\delta_g,\delta_h \in [-\delta_0\alpha_{\cal P},\alpha_{\cal P}]$ where
$ \alpha_{\cal P}= \prod_{p\in \cal P}  \frac{p-1}{p+1}$. 

 Now $|(1-C_{\cal P} )\delta_1\delta_2\delta_3|\leq |1-C_{\cal P} | \alpha_{\cal P}^3\leq\tfrac{32}{135}$ which is attained with $\cal P=\{ 2\}$.
 If \eqref{7.4} is $>\kappa$ then $\min(1-\delta_f,1-\delta_g,1-\delta_h)>1.57$, so each of $\delta_f,\delta_g,\delta_h$ are negative. We can then show that the expression in \eqref{7.4}
 is maximized, (considering the cases where $1-C_{\cal P}$ is positive or negative separately)  when 
 $\delta_f=\delta_g=\delta_h=-\delta_0$.

 For the second part of Corollary  \ref{cor2}, we must minimize
 \begin{equation}
  \tfrac 18 \left(    (1+\delta_f)  (1+\delta_g) (1+\delta_h)+ ( C_{\cal P}-1 )\delta_f\delta_g\delta_h\right) +o(1) . \label{5.15}
\end{equation}
In the smallest solution there cannot be two positive $\delta$'s, else we just replace them both by their negative.
If just one $\delta$ is positive then $( C_{\cal P} -1)\delta_1\delta_2\delta_3<0$, else replacing $\delta_i$ by $-\delta_i$ makes both terms smaller.
But then $C_{\cal P}<1$ and thus we optimize when $\cal P=\{ 2\}$. Therefore the quantity inside the brackets in  \eqref{5.15} is 
$\geq (1-\tfrac 13\delta_0 )^2 - \tfrac{32}{5\cdot 3^3}\delta_0^2>\tfrac 12> (1- \delta_0)^3$. Hence all the $\delta$'s must be negative, and  if \eqref{5.15} is 
$>\tfrac 18(1- \delta_0)^3 $ then $C_{\cal P}>1$. But then $2,3\in C_{\cal P}$, and so  the quantity inside the brackets in \eqref{5.15} is 
$\geq     (1-\delta_0/6)^3 - \tfrac{83}{25\cdot6^3}\delta_0^3 >0.7> (1- \delta_0)^3$.  Therefore the minimum occurs when $\delta_1=\delta_2=\delta_3=-\delta_0$. \qed

\medskip

\subsection*{Further calculations} 
Oleksiy Klurman asked about the minimum and maximum proportion of solutions to $a+b=c\leq x$ with 
$f(a)=\epsilon_1, f(b)=\epsilon_2,f(c)=\epsilon_3$ for given $\epsilon_1,\epsilon_2,\epsilon_3\in \{ -1,1\}$.
We solved this above when the $\epsilon_i$'s are all equal; now
we answer the remaining questions but suppress the details of the calculations, particularly the optimizations which we have seen are complicated and not very enlightening.

If $f(n)=1$ for all $n\geq 1$ then we get a proportion $0$ whenever some $\epsilon_j=-1$. Combining this observation with Corollary \ref{cor2},
Klurman's question reduces to asking for the maximum proportion when  the $\epsilon_i$ are not all equal. The formulas above imply that the proportion
of  $a+b=c\leq x$ with $f(a)=\epsilon_f, g(b)=\epsilon_g, h(c)=\epsilon_h$ where  $\epsilon_f,  \epsilon_g,  \epsilon_h\in \{ -1,1\}$ is
\[
  \tfrac 18 \left(    (1+\epsilon_f\alpha_{\cal P}\lambda_f)  (1+ \epsilon_g\alpha_{\cal P}\lambda_g) (1+ \epsilon_h\alpha_{\cal P}\lambda_h)-  (1-C_{\cal P}) \epsilon_f \epsilon_g  \epsilon_h\alpha_{\cal P}^3 \lambda_f\lambda_g\lambda_h\right) 
\]
where $\delta_f=\alpha_{\cal P}\lambda_f, \delta_g=\alpha_{\cal P}\lambda_g, \delta_h=\alpha_{\cal P}\lambda_h$ and each
$\lambda_f,\lambda_g,\lambda_h\in [-\delta_0,1]$. Since the expression is linear in each $\lambda_*$,   the optimal value is taken when each 
$\lambda_*=-\delta_0$ or $1$. 
Since the above expression is perfectly symmetric we can take $\epsilon_g=\epsilon_f=-\epsilon_h=\epsilon \in \{ -1,1\}$ to get
\[
  \tfrac 18 \left(    (1+\epsilon\alpha_{\cal P}\lambda_f)  (1+ \epsilon\alpha_{\cal P}\lambda_g) (1- \epsilon\alpha_{\cal P}\lambda_h)+  (1-C_{\cal P})    \epsilon\alpha_{\cal P}^3 \lambda_f\lambda_g\lambda_h\right) 
\]
The maximum occurs with each $\epsilon_*\lambda_*>0$, and then with $\cal P=\emptyset$.
Thus if $\epsilon =1$ then the maximum is $ \tfrac 12 (1+\delta_0)$, and if $\epsilon =-1$ then the maximum is $ \tfrac 14 (1+\delta_0)^2$.
To summarize, $( \tfrac 12 (1+\delta_0))^\mu$ is the asymptotically maximum proportion of solutions to $a+b=c\leq x$ with 
$f(a)=\epsilon_1, f(b)=\epsilon_2,f(c)=\epsilon_3$, where $\mu:=\#\{ i: \epsilon_i=-1\}$.

Finally   if $f=g=h$ then we need to find the maximum of
\[
  \tfrac 18 \left(    1+\epsilon\alpha_{\cal P}\lambda -  \alpha_{\cal P}^2\lambda^2 -C_{\cal P}   \epsilon\alpha_{\cal P}^3 \lambda^3\right) 
\]
Since $\alpha_{\cal P}^2, |C_{\cal P}|\alpha_{\cal P}^2\leq 1$ we write the above $1 -  \alpha_{\cal P}^2\lambda^2 +\epsilon\alpha_{\cal P}\lambda(1-C_{\cal P}    \alpha_{\cal P}^2 \lambda^2)$, and this is maximized when $\lambda$ has the same sign as $\epsilon$. Thus let $t=\epsilon \lambda$, so we need to maximize
\[ \tfrac 18 \left( 1+ \alpha_{\cal P}t -  \alpha_{\cal P}^2t^2 -C_{\cal P}  \alpha_{\cal P}^3t^3\right)  \]
if $\epsilon=1$  for $\lambda=t\in [0,1]$ and 
if $\epsilon=-1$  for $\lambda=-t$ with $t\in [0,\delta_0]$.

In the first case the overall maximum occurs with $t=1$ and $\cal P=\{ 2\}$. This means that the asymptotically maximum proportion of solutions to $a+b=c\leq x$ with 
$f(a)=\epsilon_1, f(b)=\epsilon_2,f(c)=\epsilon_3$ where two $\epsilon_i$'s equal 1, the other $-1$ is given by the example
of the completely multiplicative function $f$ with$f(p)=1$ except if $p=2$, yielding a proportion $\frac 8{45}=.17777\dots$.

In the second case the overall maximum occurs with $t=\delta_0$ and $\cal P=\{ 3\}$. This means that the asymptotically maximum proportion of solutions to $a+b=c\leq x$ with 
$f(a)=\epsilon_1, f(b)=\epsilon_2,f(c)=\epsilon_3$ where two $\epsilon_i$'s equal $-1$, the other $1$ is given by the example
of the completely multiplicative function $f$ with   $f(p)=1$ except if $p=3$ or $p>x^{1/(1+\sqrt{e})}$,  yielding a proportion $.15611\dots$.

\section{Binary additive problems}

The granddaddy of all additive problems in number theory is, of course, the Goldbach problem.
Typically any problem involving two variables linked by a linear equation is considered very difficult to deal with.
However Br\"udern \cite{B08} has developed a method that works in many situations provided
certain hypotheses are fulfilled (see also Theorem 1.9 in \cite{Klu} which uses techniques related to those used here).  We shall now investigate Br\"udern's idea in our
context. As usual we evaluate
\begin{equation}
\sum_{a+b=N}  f(a) g(b) = \int_0^1  \e(-\alpha N) R_f(\alpha ,N) R_g(\alpha ,N) \d \alpha
\label{9.1}
\end{equation}
where $a,b$ are positive integers, and $f$ and $g$ are multiplicative functions taking values
inside or on the unit circle, using the circle method.  The hope is that the main part of
the integral on the right side lies on the major arcs, and the contribution on the minor
arcs is easily proved to be negligible (that is $o(N)$).  Let ${\mathfrak M}\cup {\mathfrak m}$
be the partition of  $[0,1)$ into major and minor arcs, respectively. Then, by the
Cauchy-Schwarz inequality, we have
\begin{align*}
\left| \int_{{\mathfrak m} }  \e(-\alpha N) R_f(\alpha ,N) R_g(\alpha ,N) \d \alpha \right|^2
&\leq
 \int_{\mathfrak m}  |R_f(\alpha ,N)|^2 \d \alpha \cdot \int_{\mathfrak m} |R_g(\alpha ,N)|^2 \d \alpha
\\ & \leq  N  \int_{\mathfrak m}  |R_f(\alpha ,N)|^2 \d \alpha
\end{align*}
since $\int_0^1 |R_g(\alpha ,N)|^2 \d \alpha = \sum_{n\leq N} |g(n)|^2$. Hence the contribution of the
minor arcs in \eqref{9.1} is $o(N)$ if $ \int_{\mathfrak m}  |R_f(\alpha ,N)|^2 \d
\alpha=o(N)$.

\begin{prop}\label{prop7}  Suppose that $f\in \mathcal M$ and $|f(n)|=1$ for all integers $n\geq 1$. Then we have 
$ \int_{\mathfrak m}  |R_f(\alpha ,x)|^2 \d \alpha=o(x)$ if and only if  $\int_{\mathfrak M}  |M_f(\alpha)|^2 \d \alpha=x+o(x)$. We always have $\int_{\mathfrak M}  |M_f(\alpha)|^2 \d \alpha\leq x+o(x)$. Moreover
 $\int_{\mathfrak M}  |M_f(\alpha)|^2 \d \alpha= x+o(x)$ if and only if
\begin{equation} 
\sum_{\substack{(x/Q)^\varepsilon <p\leq x }}  \frac{1-{\Re e}(f(p)\pbar(p) p^{-it})}{p} =o(1) .   \label{9.5}
\end{equation}
 \end{prop}

\begin{proof} During the proof of Lemma \ref{ArcEstimates} we saw that 
\[
\int_{{\mathfrak M}} |E_f(\alpha,x)|^2 \d \alpha =o(x) \text{ and } \int_{{\mathfrak M}} |M_f(\alpha)|^2 \d \alpha \ll x, 
\]
so that  $$\Big|\int_{\mathfrak M}   M_f(\alpha) \overline{E_f(\alpha,x)} \d \alpha \Big|  \leq 
\left( \int_{\mathfrak M}  |M_f(\alpha )|^2 \d \alpha \int_{\mathfrak M}  |E_f(\alpha,x )|^2 \d \alpha \right)^{1/2} =o(x) ,$$ 
and therefore
\begin{align*}
 \int_{\mathfrak M}  |R_f(\alpha,x)|^2 \d \alpha &- \int_{\mathfrak M}  |M_f(\alpha)|^2 \d \alpha 
\cr&\leq 
 2\left| \int_{\mathfrak M}   |M_f(\alpha) \overline{E_f(\alpha,x)} \d \alpha\right| + \int_{\mathfrak M}  |E_f(\alpha,x)|^2 \d \alpha  =o(x).
\end{align*} 
By Parseval we have
\[
 \int_0^1  |R_f(\alpha,x)|^2 \d \alpha = \sum_{n\leq x} |f(n)|^2 = x+O(1),
 \]
 and so 
 \[ \int_{\mathfrak m}  |R_f(\alpha ,x)|^2 \d \alpha=o(x) \text{ if and only if }
  \int_{\mathfrak M}  |M_f(\alpha)|^2 \d \alpha=x+o(x).
  \]
Now, by Lemma  \ref{ArcEstimates}, we have
\[
 \int_{\mathfrak M}  |M_f(\alpha)|^2 \d \alpha =  |\mu(f_*,x)|^2 x^2 \sum_{r|q\leq x/Q}  r  \frac{|\kappa(q/r)|^2}{\phi(q)}  \int_{\beta=-1/qQ}^{1/qQ} | I(x,\beta,t) |^2  \d \beta  .
 \]
 If $q\ll x/Q$ then  $x\int_{\beta=-1/qQ}^{1/qQ}  | I(x,\beta,t) |^2  \d \beta \asymp 1$
 by  \eqref{eq: Iintegralbd}; and it is $1+o(1)$ if $q=o(x/Q)$. Using this and  
 \eqref{struceq} we deduce that 
 \[
 \int_{\mathfrak M}  |M_f(\alpha)|^2 \d \alpha 
\leq  |\mu((f_*)_s,x)|^2 |\mu((f_*)_\ell,x)|^2\, x(1+o(1)) \sum_{r|q\leq x/Q}  r  \frac{|\kappa(q/r)|^2}{\phi(q)} 
 \]
with, say, $z=x/Q$.  Since $|\kappa(m))|= \prod_{p|m} |1-f_*(p)|$ this last sum is 
\[
\leq   r \sum_{\substack{m\geq 1\\ p|m \implies p\leq x/Q}}   \frac{|\kappa(m)|^2}{\phi(rm)}
= \prod_{p\leq x/Q}  \left(1-\frac{1}p\right)^{-2} \left|1-\frac{f_*(p)}p\right|^{2} \sim \frac 1{|\mu((f_*)_s,x)|^2}.
\]
Now since $|\mu((f_*)_\ell,x)|\leq 1$, we deduce that $ \int_{\mathfrak M}  |M_f(\alpha)|^2 \d \alpha
\leq (1+o(1)) x$.
In order to get equality here we must have (asymptotic) equality in each of the last few steps. In particular we must have $|\mu((f_*)_\ell,x)|\sim 1$, and so 
\[
 \sum_{\substack{ x/Q<p\leq x }}  \frac{1-{\Re e}(f_*(p)) }{p} =o(1) 
 \]
by Hal\'asz's Theorem.   In this case we deduce that $ \int_{\mathfrak M}  |M_f(\alpha)|^2 \d \alpha\sim x$ if and only if 
\[
    \sum_{1\leq m= o(x/Qr)}  \frac{r}{\phi(mr)} \prod_{p|m} |1-f_*(p)|^2   \sim \frac 1{|\mu((f_*)_s,x)|^2},
 \]
 since the larger $q=mr$ in the sum are weighted by an $I$-integral that is $<1$. One can show that this happens if and only if the primes $>(x/Q)^\varepsilon$ do not make a significant contribution; that is, 
\[
 \sum_{\substack{(x/Q)^\varepsilon <p\leq x/Q }}  \frac{1-{\Re e}(f_*(p)) }{p} =o(1) 
 \]
\end{proof}

As a corollary we can recover the hypothesis \eqref{Brud} used by  Br\"udern \cite{B08}:

\begin{cor}\label{cortoprop7}  Suppose that $f\in \mathcal M$ and $|f(n)|=1$ for all integers $n\geq 1$. Then we have 
$ \int_{\mathfrak m}  |R_f(\alpha ,x)|^2 \d \alpha=o(x)$ for all sufficiently large $x$  if and only if 
\begin{equation} \label{Brud}
\sum_{\substack{p  }}  \frac{1- {\Re e}(f(p)\pbar(p) p^{-it})}{p} \ll 1  
\end{equation}
\end{cor}

If \eqref{9.5} holds then we can evaluate the main term in
certain binary additive problems, like \eqref{9.1}  to obtain
\begin{equation}
\sum_{a+b=N}  f(a) g(b) = \int_{{\mathfrak M} }   \e(-\alpha N) M_f(\alpha ,N) M_g(\alpha ,N) \d \alpha +o(N)
\end{equation}
or
\begin{equation}
\sum_{n\leq N}  f(n) g(n+1) = \int_{{\mathfrak M} } \e(\alpha) M_f(\alpha ,N) M_g(-\alpha ,N) \d \alpha +o(N)
\end{equation}
using only the major arcs.  
 The actual asymptotics can be obtained without using the circle method, as in \cite{Klu}, Corollary 1.4, so we will not give explicit details here.
 
\goodbreak
\section{Upper bounds on exponential sums} \label{MVBound}
In this section we give explicit upper bounds on $R_f(\alpha,x)$.

\subsection{The conjecture  \eqref{1.2} can only fail in very specific circumstances} We sketch a proof of the following:

\begin{prop} \label{NewBd} If   \eqref{1.2} does not hold then \eqref{qrange} holds.
\end{prop}

 \begin{proof} [Proof Sketch] Let $\varepsilon>0$ be fixed.  Let $Q_1$ be defined by \eqref{defQ1} and
$Q_2=x/(\log x)^3$.  There exists $q\leq Q_2$ such that $|\alpha-a/q|\leq 1/qQ_2$ for some integer $a$. 

If $Q_1<q\leq Q_2$ then \eqref{1.2} follows from \eqref{1.1}.

If $q\leq Q_1$ and $|\alpha-a/q|>(\log q\log\log q)/x=:1/qQ_3$, then there exists $r\leq Q_3$ such that $|\alpha-b/r|\leq
1/rQ_3$ for some integer $b$, and hence $r\ne q$.  Therefore, as $Q_3>2q$,
$$
\frac 1 {qQ_2} + \frac 1 {2qr} > \frac 1 {qQ_2} +  \frac 1 {rQ_3}  \geq |\alpha-a/q| +
|\alpha-b/r| \geq |a/q-b/r|\geq  \frac 1 {qr},
$$
which implies that $r>Q_2/2>Q_1$ and  then  \eqref{1.2}  follows from \eqref{1.1}.

If $q\leq Q_1$ and $|\alpha-a/q|\leq (\log q\log\log q)/x=1/qQ_3$ then we apply Theorem \ref{thm1}.
As $S_{f_j}(x)\ll (\phi(r)/r)x$, the $j$th term is 
\[
\ll            \prod_{p^a\|q, p\nmid r}  \frac{2p}{(p-1)p^{a/2}} \prod_{p^a\|q, \ p^b\| r}  \frac{p^{b/2}}{p^{a/2}}  \cdot \frac x{\sqrt{q}}
\ll   \frac x{\sqrt{q}} .
\]
Moreover in this range the error term is, taking $J>1/\varepsilon^2$,
\[  
\ll   \frac{ x}{ (\log x)^{ 1 - \varepsilon}} + \frac x{\sqrt{q}}      .
\] 
This is $\ll x/\sqrt{q}$ for $q\leq (\log x)^{ 2 - \varepsilon}$. 
\end{proof}

\subsection{Exponential sums over friable numbers} \label{friBd}
We sketch the minor changes needed for the following  slightly stronger estimate than Proposition 1 of~\cite{dlB98}: \ 
If  $|\alpha-a/q|\leq 1/q^{2}$ with $(a,q)=1$ and $q\leq x$ then
\begin{equation}\label{1.5} 
  \sum_{\colt{n\leq x}{P(n)\leq y}}  f(n) \e(\alpha n) \ll 
\sqrt{xy}  + \left( \frac{x}{\sqrt q} + \sqrt{xq\log (2x/q)}\right) \log y  + \frac x
{  \e^{ \{ \frac 12 +o(1) \} \sqrt{\log x \log \log x} }} . 
\end{equation}

\begin{proof}
We follow the proof of Proposition 1 of \cite{dlB98} (though rectify the omission of the $h=0$ term at the end of the long display of equations at the top of page 62), but now choosing
$T=\exp (  \sqrt{\log x \log \log x}) $. The key change that we make is to give the (easily proved) more precise upper bound
$$
\sum_{1\leq h\leq x/K}  \min \left( K, \frac 1{\| \theta h \|}\right) \ll K+ q \log (2x/q) + \frac xq + \frac xK \log 2K
$$
(rather than the typographically simpler  $\ll (K+q+x/q+x/K)\log x$).  The only other change is that we note that the sum of $(xK)^{1/2}$, over $K$ of the form $2^jT$ with $j$ an integer, and $2^jT\leq y$,
is $\ll (xy)^{1/2}$.
\end{proof}

 It would be of interest to remove further logs from the right-hand side of \eqref{1.5}: For instance, based on the case $y=x$ we believe, based on \eqref{1.1}, that we should have something like
$\sqrt{xy}/\log y$ in place of the $\sqrt{xy}$ term.

\end{document}